\documentclass[11pt,a4paper]{amsart}
\usepackage{amscd}
\usepackage{amsmath}
\usepackage{latexsym}
\usepackage{amsfonts}
\usepackage{amssymb}
\usepackage{amsthm}
\usepackage{graphicx}
\usepackage{verbatim}
\usepackage{mathrsfs}
\usepackage{enumerate}
\usepackage{color}

\newtheorem{definition}{Definition}[section]
\newtheorem{theorem}{Theorem}[section]
\newtheorem{lemma}[theorem]{Lemma}
\newtheorem{corollary}[theorem]{Corollary}
\newtheorem{proposition}[theorem]{Proposition}
\newtheorem{remark}[theorem]{Remark}

\newcommand{\RR}{\mathbb{R}}

\DeclareMathOperator{\Div}{div}

\allowdisplaybreaks
\begin{document}

\title[Global existence of uniformly locally energy solution]{Global existence of uniformly
locally energy solution for the incompressible fractional Navier-Stokes equations}

\author[J. Li]{Jingyue Li}
\address[J. Li]{The Graduate School of China Academy of Engineering Physics, Beijing 100088, P.R. China}
\email{m\_lijingyue@163.com}

\date{\today}

\begin{abstract}
In this paper, we introduce the concept of local Leray solutions starting from a
locally square-integrable initial data to the fractional
Navier-Stokes equations with $s\in [3/4,1)$. Furthermore, we prove its local in time existence when $s\in (3/4, 1)$. In particular,  if the locally square-integrable initial data vanishs at
infinity, we show that  the fractional
Navier-Stokes equations  admit a global-in-time local
Leray solution  when $s\in [5/6, 1)$. For such local
Leray solutions starting from locally square-integrable initial data vanishing at
infinity, the singularity only occurs in $B_R(0)$ for some $R$.
\end{abstract}

\keywords{Fractional Navier-Stokes equations, nonlocal effect, local Leray solutions.}
\maketitle
\section{Introduction}
\setcounter{section}{1}
\setcounter{equation}{0}
In this paper, we consider the following incompressible fractional Navier-Stokes equations
in $\RR^3\times (0,\infty)$ with $s>0$:
\begin{equation}\tag{FNS}\label{NS}
\partial_t u+\Lambda^{2s}u+u\cdot\nabla u+\nabla P=0,\quad \Div u=0,
\end{equation}
equipped with the initial data
\begin{equation}\tag{FNSI}\label{NSI}
u(0,x)=u_0(x) \quad \text{in} \quad\RR^3.
\end{equation}
Here the vector field $u$ and the scalar function $P$ describe the velocity field and the
associated pressure, respectively. The fractional Laplacian operator $\Lambda^s$
is a non-local operator defined in terms of the Fourier transform:
\[\mathcal{F}(\Lambda^su)(\xi)=|\xi|^s\hat{u}(\xi).\]

The fractional Laplacian operator appears in a wide class of hydrodynamics, statistical mechanics, physiology, including L\'{e}vy flights, stochastic interfaces and anomalous diffusion problems, see \cite{Woy01,Res95}. So the fractional Navier-Stokes system \eqref{NS} has very important physical significance. For example, when $s\in (0,1)$, Zhang in \cite{Zha12} described the stochastic Lagrangian particle approach via \eqref{NS}.
In particular, when $s=1$, the system \eqref{NS} becomes the classical 3D incompressible Navier-Stokes equations.

Recently, there are increasing interesting results about weak and strong solutions of the fractional Navier-Stokes equations \eqref{NS}. The strong solutions of the fractional
Navier-Stokes equations \eqref{NS}-\eqref{NSI} have been studied in a series of work using analytical tools, see e.g. \cite{CMW08,KMX08}. In this paper, we focus on weak solutions to \eqref{NS}. This work can be traced back to Leray
\cite{Ler34}. In \cite{Ler34}, Leray introduced the concept of Leray-Hopf weak solutions to \eqref{NS}-\eqref{NSI} with $s=1$ and then showed the  global existence for each divergence free initial data $u_0\in L^2(\RR^3)$. Following
the argument in \cite{Ler34}, for all $s>\frac12$, we can prove that for any solenoidal vector $u_0\in L^2(\RR^3)$, system \eqref{NS}-\eqref{NSI} admits at least one weak solution in
$L^{\infty}([0,\infty);L^2(\RR^3))\cap L^2([0,\infty);\dot{H}^s(\RR^3))$, which satisfies the continuity at $t=0$ in $L^2(\RR^3)$ and
the energy estimate:
\begin{equation}\label{G-ene}
\int_{\RR^3}|u(t)|^2\,\mathrm{d}x+2\int^t_0\int_{\RR^3}|\Lambda^s
u|^2\,\mathrm{d}x\mathrm{d}\tau
\leq \int_{\RR^3}|u_0|^2\,\mathrm{d}x
\end{equation}
For sake of simplicity, such weak solutions are also called Leray-Hopf weak solutions. As for $s\geq 5/4$, Lions \cite{Lio69} proved that \eqref{NS}-\eqref{NSI}
admits a unique global smooth solution for any prescribed smooth initial data. However, for $s<5/4$, the uniqueness and global regularity of Leray-Hopf weak solutions are
still open problems. Instead of it, the study of partial regularity theory has been put on the agenda. In this respect, Scheffer
\cite{Sch76,Sch77,Sch80} first discussed a class
of weak solutions to classical Navier-Stokes equations, which satisfy the following local energy
inequality instead of \eqref{G-ene}:
\begin{equation}\label{L-ene}
\begin{split}
2\int^{\infty}_0\int_{\RR^3}|\nabla u|^2\varphi\,\mathrm{d}x\mathrm{d}\tau\leq&
\int^{\infty}_0\int_{\RR^3}\Big(|u|^2(\partial_t\varphi+\Delta\varphi)+(|u|^2+2p)u\cdot\nabla\varphi\Big)\,\mathrm{d}x
\mathrm{d}\tau
\end{split}
\end{equation}
for any nonnegative function $\varphi\in \mathcal{D}(\RR^3\times\RR^+)$. He proved such
weak solution has a singular set with finite $5/3$-Hausdorff measure. Later, Scheffer's
result was improved by Caffarelli, Kohn and Nirenberg in \cite{CKN82}, where they
introduced the concept of suitable weak solutions and proved that the 1D parabolic Hausdorff measure of the associated singular set is zero. For a
simplified proof, see Lin \cite{Lin98}.  For $1<s<5/4$, Katz and Pavlovi\'{c}  \cite{KP02} proved the first version of the
CKN theory, that is, the Hausdorff dimension of singular set at the first blow up time
is at most $5-4s$. Lately, Colombo, De Lellis and Massaccesi in \cite{CDM17} introduced the
definition of suitable weak solution via the harmonic extension of the higher order fractional
Laplace operator in \cite{Ya13} and proved a stronger version of the results of Katz and
Pavlovi\'{c}, which fully extends the CKN theory. For $3/4<s<1$, via
the harmonic extension established in \cite{CS07}, Tang and Yu in \cite{TY15} gave the definition and
existence of suitable weak solutions to \eqref{NS} and then generalized CKN theory to this case. The CKN theory at endpoint case $s=3/4$ was studied by Ren, Wang and Wu in \cite{RWW16}.

To give a precise statement about the question discussed in this paper, we introduce some Banach spaces in local measure sense. For any $p\in [1,\infty]$, we define
\begin{align*}
&L^p_{uloc}(\RR^n)\triangleq\{f\in L^p_{loc}(\RR^n)\,|\,\|f\|_{L^p_{uloc}(\RR^n)}<\infty\},\quad\|f\|_{L^p_{uloc}(\RR^n)}
\triangleq\sup_{x_0\in\RR^n}\|f\|_{L^p(B_1(x_0))}.\\
&E^p(\RR^n)\triangleq\{f\in L^p_{uloc}(\RR^n)\,|\,\lim_{|x_0|\to\infty}\|u\|_{L^p(B_1(x_0))}=0\}.
\end{align*}
It is obvious that $E^p$ is the closure of $C^{\infty}_0(\RR^n)$ under $L^p_{uloc}(\RR^n)$ norm when $1\leq p<\infty$. Let $\phi\in \mathcal{D}(\RR^n)$ be a nonnegative function such that $\phi=1$ for all $|x|\leq 1$ and $\phi=0$ for all $|x|\geq 2$. Set $\phi_{x_0}(\cdot)=\phi(x_0-\cdot)$, $x_0\in \RR^n$.
Then for all $s\in \RR$, $H^s_{uloc}(\RR^n)$ and $\bar{H}^s_{uloc}$ can be defined as follows:
\begin{align*}
&H^s_{uloc}(\RR^n)\triangleq\big\{u\in
H^s_{loc}(\RR^n)\,\big|\,\|u\|_{H^s_{uloc}(\RR^n)}<\infty\big\},\,\|u\|_{H^s_{uloc}(\RR^n)}
\triangleq\sup_{x_0\in
\RR^n}\|\phi_{x_0}u\|_{H^s(\RR^n)}.\\
&\bar{H}^s_{uloc}(\RR^n)\triangleq\big\{u\in
H^s_{uloc}(\RR^n)\,\big|\,\lim_{|x_0|\to\infty}\|\phi_{x_0}u\|_{H^s(\RR^n)}=0\}.
\end{align*}
Finally, we define the mixed time-space Banach space:
\begin{align*}
&L^{p,q}_{uloc}(T_1,T_2,\RR^n)\triangleq\{f\in
L^p([T_1,T_2];L^q_{loc}(\RR^n))\,|\,\|u\|_{L^{p,q}_{uloc}(T_1,T_2,\RR^n)}<\infty\},\\
&G^{p,q}(T_1,T_2,\RR^n)\triangleq\{f\in L^{p,q}_{uloc}(T_1,T_2,\RR^n)\,|\,\lim_{|x_0|\to +\infty
}\|u\|_{L^p([T_1,T_2];L^p(B_1(x_0)))}=0\}.
\end{align*}
with $\|f\|_{G^{p,q}(T_1,T_2,
\RR^n)}\triangleq\|f\|_{L^{p,q}_{uloc}(T_1,T_2,\RR^n)}\triangleq\sup_{x_0\in \RR^n}\|u\|_{L^p([T_1,T_2];L^q(B_1(x_0)))}$.
For brevity, we adapt the following notations:
\begin{align*}
&L^p\triangleq L^p(\RR^3),\,H^s\triangleq H^s(\RR^3),\, L^p_{uloc}(\RR^3)\triangleq L^p_{uloc},\, E^p(\RR^3)\triangleq E^p,\\
&H^s_{uloc}\triangleq H^s_{uloc}(\RR^3),\,\bar{H}^s\triangleq \bar{H}^s_{uloc}(\RR^3),\,L^{p,q}_{uloc}(0,T,\RR^3)\triangleq L^{p,q}_{uloc}(T),\,G^{p,q}(0,T,\RR^3)\triangleq G^{p,q}(T).
\end{align*}
For $s=1$, with the aid of $\varepsilon$-regularity theory established in \cite{CKN82},
Lemari\'{e}-Rieusset in \cite{LR02} introduced a concept of local Leray solutions to
\eqref{NS}, which satisfies local energy inequality \eqref{L-ene}, and proved that for any
divergence free initial data $u_0$ in $E^2$, \eqref{NS} admits a global in time local Leray
solution.

Based on the partial regularity results given in \cite{CDM17,RWW16,TY15}, a nature and
interesting question arises:

\textbf{(Q):} \textit{Does there exist a global in time weak solution to
\eqref{NS}-\eqref{NSI} with $s\in (3/4, 1)$ for each divergence free initial
data $u_0$ in $E^2$?}\\
Compared with $s=1$, to solve this question for the general case $s\neq 1$, we face a difficulty point which
does not appear in classical Navier-Stoke equations $(s=1)$.

Difficult point: Since the fractional operator $\Lambda^s$ is an nonlocal operator, we can't deal
with $\Lambda^{2s}u$ as classical Navier-Stokes equations. To overcome it, we
adapt the nonlocal commutator $[\Lambda^s,\psi]$ where $s\in (0,1)$ and $\varphi\in C^1(\RR^n)$. With it, when $s\in (0,1)$, formally, we directly write that
\[\langle \Lambda^{2s}u,u\varphi\rangle=\langle \varphi\Lambda^{2s}u,u\tilde{\varphi}\rangle=\langle [\varphi,\Lambda^{s}]\Lambda^{s}u,u\tilde{\varphi}\rangle+\langle\varphi\Lambda^{s}u,
[\Lambda^s,\tilde{\varphi}]u\rangle+\langle\varphi \Lambda^{s}u,\varphi \Lambda^{s}u\rangle\]
for any $\tilde{\varphi}\in C^1(\RR^n)$ satisfying $\varphi\tilde{\varphi}=1$.
While for $s>1$, due to that the commutator $[\Lambda^s,\varphi]$ fails to be controlled, we will first adapt the following technique:
\[\langle\Lambda^{2s}u,u\varphi\rangle=\langle(-\Delta)\Lambda^{2s-2}u,u\varphi\rangle
=-\langle\Lambda^{2s-2}\nabla u,\nabla (u\varphi)\rangle\]
and then  use the commutator  again. 
The work on $1<s\leq 5/4$ may need more complicated calculations. Hence, in this paper, we only consider the case $s<1$.

With the help of $[\Lambda^s,\varphi]$, we can define the \textit{local Leray solutions} to \eqref{NS}-\eqref{NSI} with $s<1$.
\begin{definition}[Local Leray solutions]\label{def.ller}
Let $s\in [\frac34,1)$ and $u_0$ be a solenoidal vector field in $L^2_{uloc}$. We call $u$
is a local Leray solution to \eqref{NS} on $\RR^3\times (0,T)$ starting from $u_0$ if $u$
satisfies the following conditions: 
\begin{enumerate}[\rm(i)]
  \item $u\in L^{\infty}([0,T'];L^{2}_{uloc})$ and $\Lambda^su\in L^{2,2}_{uloc}(T')$ for any
      $T'<T$;
  \item $\exists P\in \mathcal{D}'(\RR^3\times (0,T))$, $\partial_t u+\Lambda^{2s}
      u+u\cdot\nabla u+\nabla P=0$ in the sense of distribution;
  \item for any compact subset $K$ of $\RR^3$, $\lim_{t\to 0+}\|u-u_0\|_{L^2(K)}=0$;
  \item $u$ is suitable, that is, for any nonnegative function $\phi\in
      \mathcal{D}(\RR^3\times \RR^+)$, $u$ satisfies the following local energy
      inequality
      \begin{equation}\label{eq.loc}
      \begin{split}
       2\int^t_0\int_{\RR^3} \psi\big|\Lambda^s u\big|^2\,\mathrm{d}x\mathrm{d}\tau\leq
       &\int^t_0\int_{\RR^3}\big|u\big|^2\partial_t\psi\,\mathrm{d}x\mathrm{d}\tau
       -2\int^t_0\int_{\RR^3} [\tilde{\psi},\Lambda^{s}]\Lambda^s u\cdot
       u\psi\,\mathrm{d}x\mathrm{d}\tau\\
       &-2\int^t_0\int_{\RR^3}
       [\Lambda^s,\psi]u\cdot(\tilde{\psi}\Lambda^su)\,\mathrm{d}x\mathrm{d}\tau\\
       &+\int^t_0\int_{\RR^3}|u|^2
       u\cdot\nabla\psi\,\mathrm{d}x\mathrm{d}\tau
       +2\int^t_0\int_{\RR^3}P\cdot u\nabla\psi\,\mathrm{d}x\mathrm{d}\tau
      \end{split}
      \end{equation}
      with $\nabla P=-\nabla \frac{\Div\Div}{\Delta}(u\otimes u)$. Here $\tilde{\psi}\in \mathcal{D}(\RR^3)$ satisfies $\tilde{\psi}\geq 0$ and $\tilde{\psi}\psi=\psi$.
\end{enumerate}
\end{definition}
\begin{remark}\label{rem.1}
The restriction on the pressure $P$ avoids the existence of trival local Leary solution, for example, $u=\rho(t)$ and $P=-\rho'(t)\cdot x$ where $\rho(t)$ is a smooth vector field with $\rho(t)\neq 0$. In addition, for all $x_0\in\RR^3$ and $r>0$, we can decompose $P$ in $B_{r}(x_0)$ as $P_{x_0,r} +P_{x_0,r} (t)$ where
$P_{x_0,r}(t)$ is a function depending only on $x_0,r$ and 
\begin{equation}\label{pess}
\begin{split}
P_{x_0,2}(x,t)=& -\Delta^{-1}\Div\Div (u(t)\otimes
u(t)\psi_{x_0})\\
&-\int_{\RR^3}\big(k(x-y)-k(x_0-y)\big)(u\otimes
u)(y,t)(1-\psi_{x_0}(y))\,dy\\
\triangleq&
P^1_{x_0,2}(x,t)+P^2_{x_0,2}(x,t)
\end{split}
\end{equation}
with $\psi_{x_0}\in\mathcal{D}(B_{4r}(x_0))$ satisfying $0\leq \psi_{x_0}\leq 1$ and $\psi=1$ in $B_{2r}(x_0)$. Here $k(x)$ is the kernel of the Calderon-Zygmund operator $\Delta^{-1}\Div\Div$, satisfying
\begin{equation}\label{eq.press1}
|k(x-y)-k(x_0-y)|\lesssim \frac{|x-x_0|}{|x-y|^4},\quad \text{if}\quad |x-x_0|\leq
\frac12|x-y|.
\end{equation}
\end{remark}
\begin{remark}
The restriction on $s$ ensures that the following term is well defined:
\[\int^t_0\int_{\RR^3}|u|^2u\cdot\nabla\psi\,\mathrm{d}x\mathrm{d}\tau.\]
\end{remark}
By regularization method, we can prove our first main result.
\begin{theorem}\label{TH1}
Let $s\in (3/4,1)$. Then, for each solenoidal vector field $u_0\in L^2_{uloc}$, there exists a $T>0$ and a local Leary
solution to \eqref{NS} starting from $u_0$ on $\RR^3\times (0,T)$.
\end{theorem}

Due to the last two terms of the right side of \eqref{eq.loc}, we can't
directly get the global-in-time existence of local Leray solution $u$. Inspired by \cite{LR02}, we can construct a local Leray solution on $(0,\infty)$ via $\varepsilon$-regularity theory in \cite{TY15} and ``weak-strong" uniqueness theorem. To realize it, we first show the following additional regularities of local Leray solutions with initial data in $E^2$.
\begin{theorem}\label{TH1-1}
If $u$ is a local Leray solution to \eqref{NS} on $\RR^3\times (0,T)$ starting from
$u_0\in E^2(\RR^3)$ with $\Div u_0=0$, then $u\in L^{\infty}([0,T'];E^2)$ and $\Lambda^s u\in G^{2,2}(T')$ for any $T'\in(0,T)$ and
 \begin{equation}\label{est-1-1}
 \lim_{s\to t+}\|u(s)-u(t)\|_{L^2_{uloc}}=0\quad\text{for }t=0\,\,\text{or almost every }t\in (0,T).
 \end{equation}
\end{theorem}
With the help of Theorem \ref{TH1-1}, we say that \eqref{eq.loc} is equivalent to the local energy inequality given in \cite{TY15} via the harmonic extension. This implies that the local Leray solution with initial data in $E^2$ is a suitable weak solution defined in Section \ref{sec.4}. For details, see Section \ref{sec.4}. Hence, by Proposition \ref{pro.CKN}, we can prove Proposition \ref{pro.regu}, which says that the local Leray
solution $u$ starting from $u_0\in E^2$ is regular in  $B^c_R$ away from $t=0$ for some large enough $R$ and there exists a $t_0$ closed to $T$ such that $u(t_0)\in \bar{H}^{s}_{uloc}$ and $\lim_{t\to t_0+}\|u(t)-u(t_0)\|_{E^2}=0$. Based on these, invoking the well-posedness theory in Theorem \ref{well-loc} and the weak-strong uniqueness of local Leray solutions in Proposition \ref{pro.uni}, we can prove the following theorem via
the construction method in \cite{LR02}, which gives a positive answer to \textbf{(Q)}.
\begin{theorem}\label{TH2}
Let $s\in [5/6, 1)$. Then, for any solenoidal vector field $u_0\in E^2$, there exists a local Leray solution $u$ to
\eqref{NS}-\eqref{NSI} in $\RR^3\times (0,\infty)$.
\end{theorem}
\begin{remark}
Since $\bar{H}^s_{uloc}\hookrightarrow \bar{H}^{\alpha}_{uloc}$ for all $0\leq \alpha\leq s$, to obtain a solution in $\bar{H}^{\alpha}_{uloc}$ for some $\alpha$ which ensures that the uniqueness in Proposition \ref{pro.uni} holds, we need $s\geq 5/2-2s$, where $5/2-2s$ is the critical index. This gives that $s\geq 5/6$.
\end{remark}
\begin{remark}
From the above analysis, we have that for each local Leray solution $u$ to \eqref{NS}-\eqref{NSI} on $\RR^3\times (0,T)$ with $u_0$ in $E^2$, it is  regular in $B^c_R$ away from $t=0$ for some large enough $R>0$.
\end{remark}
\begin{remark}
Given $\lambda>0$, set
\[u_{\lambda}(x,t)=\lambda^{2s-1}u(\lambda x,\lambda^{2s}t),\quad P_{\lambda}(x,t)=\lambda^{4s-2}P(\lambda x,\lambda^{2s}t).\]
We call a solution on $\RR^3\times (0,\infty)$ is a forward self-similar solution if and only if $u(x,t)=u_{\lambda}(x,t)$ and $P(x,t)=P_{\lambda}(x,t)$ for all $\lambda>0$. The existence of forward self-similar solutions to \eqref{NS}-\eqref{NSI} was firstly studied in the case $s=1$, see e.g. \cite{Can95,CP96}. Lately, Lai, Miao and Zheng in \cite{LMZ17} proved the existence of forward self-similar solutions to \eqref{NS}-\eqref{NSI} with $5/6<s\leq 1$ for arbitrary large self-similar initial data via the blow-up argument. Since $E^2$ contains non-trivial scale-invariant functions, for example: $\frac{\sigma(x)}{|x|^{2s-1}}$, inspired by \cite{JS14}, we think that, with the help of the global-in-time existence of local Leray solutions in Theorem \ref{TH2}, the result given in \cite{LMZ17} may be proved in another way.
\end{remark}
The rest of this paper is structured as follows. In Section \ref{sec.2}, we give some key foundation estimates, which will be used repeatedly in following sections. In Section \ref{sec.3}, by
regularization method, we prove Theorem \ref{TH1}. Section \ref{sec.4}
is devoted to developing decay and regularity properties
of local Leray solutions to \eqref{NS}-\eqref{NSI} with initial data in $E^2$.
In Section \ref{sec.5}, we complect the proof of Theorem \ref{TH2} via Theorem \ref{TH1} and the
properties given in Section 4. Finally, in Appendix, for the convenience of readers, we prove some results relevant to FNS, which are
recognized by everyone but not proven before.

\subsection*{Notation}
We denote $C$ as an absolute positive constant. $C_{\lambda, \gamma,\cdots}$ denotes a positive constant depending only on $\lambda, \gamma, \cdots$. We adopt the convention that nonessential constant $C$ may change from line to line. Given two quantities $a$ and $b$, we denote $a\lesssim b$ and $a\lesssim_{\lambda,\gamma,\ldots} b$ as $a\leq C b$ and $a\leq C_{\lambda,\gamma,\cdots}\, b$ respectively. In addition, if $Cb\leq a\leq C^{-1}b$ or $C_{\lambda,\gamma,\ldots}b\leq a\leq C_{\lambda,\gamma,\ldots}^{-1}b$, we say $a\simeq b$ or $a\simeq_{\lambda,\gamma,\ldots} b$. For any $x_0\in\RR^3$ and $t_0\in \RR^+$, $B_R(x_0)$ means a ball in $\RR^3$ with radius $R$ centered at $x_0$, $B^c_{R}(x_0)=\RR^3\setminus B_R(x_0)$ and $Q_R(x_0,t_0)=B_R(x_0)\times (t_0-R^2,t_0)\in \RR^3\times \RR$.

\section{Preliminaries}\label{sec.2}
\setcounter{section}{2}
\setcounter{equation}{0}

In this section, we mainly introduce some important results which will be used in the
following sections.
\subsection{Estimates of Heat kernel and Oseen Kernel of fractional Laplacian}
In this subsection, we consider the following Cauchy problem for the linear fractional Stokes problem in
$\RR^n\times\RR^+$:
\begin{equation*}
\begin{cases}
\partial_t u+\Lambda^{2s}u+\nabla p=\Div F,\quad \Div u=0,\\
u(x,0)=u_0,
\end{cases}
\end{equation*}
where $F$ is a given second-order tensor filed. Since $\Div u=0$, applying Leray project
operator $\mathbb{P}=Id+\mathcal{R}\otimes \mathcal{R}$ where
$\mathcal{R}\triangleq(\mathcal{R}_1,\ldots,\mathcal{R}_n)$ and $\mathcal{R}_j$ is the Riesz operator, to the above first equation, by Duhamel's formula, we get
that
\begin{align*}
u(x,t)=e^{-t\Lambda^{2s}}u_0+\int^t_0 e^{-(t-\tau)\Lambda^{2s}}\mathbb{P}(\nabla\cdot
F)\,\mathrm{d}\tau=G_t\ast u_0+\int^t_0 \mathcal{O}_{j,k,t-\tau}\ast(\partial_iF_{ik})(\tau)\,\mathrm{d}\tau
\end{align*}
where $G_t(x)$ and $\mathcal{O}_{j,k,t}$ are the kernel functions of $e^{-t\Lambda^{2s}}$
and $e^{-t\Lambda^{2s}}(\delta_{jk}+\mathcal{R}_j\mathcal{R}_k)$, respectively. It's obvious that
\[\Lambda^{\alpha}G_t(x)=t^{-\frac{\alpha}{2s}}t^{-\frac{n}{2s}}(\Lambda^{\alpha}G_1)
(\tfrac{x}{t^{1/(2s)}}),\quad
\Lambda^{\alpha}\mathcal{O}_{j,k,t}(x)=t^{-\frac{\alpha}{2s}}t^{-\frac{n}{2s}}
(\Lambda^{\alpha}\mathcal{O}_{j,k,1})(\tfrac{x}{t^{1/(2s)}})\]
and $\int_{\RR^n}G_t(x)\mathrm{d}x=1$. In addition, $G_t$ and $\mathcal{O}_{j,k,t}$ satisfy
the following point-wise and $L^p-L^q$ estimates.
\begin{lemma}\label{lem.G_t}
Let $s>0$. Then, there exists a absolute constant $C$ such
that for all $x\in \RR^n$
\begin{equation}\label{G_t-p}
|G_1(x)|\leq C(1+|x|)^{-n-2s},\quad|\Lambda^{\alpha}G_1|\leq
C(1+|x|)^{-n-\alpha}\,(\alpha>0).
\end{equation}
Moreover,  we have that for all $0\leq \alpha_2\leq \alpha_2$ 
\begin{align}
&\|\Lambda^{\alpha_1}G_t\ast f\|_{L^p(\RR^n)}\leq C_{s,\alpha,p,r}t^{-\frac{\alpha_1-\alpha_2}{2s}-\frac n{2s}(\frac
1r-\frac 1p)}\|\Lambda^{\alpha_2}f\|_{L^r(\RR^n)},\quad 1\leq r\leq p\leq \infty,\label{G_t-Lp'}\\
&\|\Lambda^{\alpha_1}G_t\ast f\|_{L^r_{uloc}(\RR^n)}\leq C_{s,\alpha}t^{-\frac
{\alpha_1-\alpha_2}{2s}}\|\Lambda^{\alpha_2}f\|_{L^r_{uloc}(\RR^n)},\quad1\leq r\leq \infty,\label{G_t-Lp}\\
&\|\Lambda^{\alpha_1}G_t\ast f\|_{L^{\infty}(\RR^n)}\leq C_{s,\alpha}t^{-\frac
{\alpha_1-\alpha_2}{2s}}\max\{1,t^{-\frac3{2ps}}\}\|\Lambda^{\alpha_2}f\|_{L^r_{uloc}(\RR^n)},\quad 1\leq r< \infty.\label{G_t-Lp''}
\end{align}
\end{lemma}
To prove it, we use the following estimate:
\begin{lemma}\cite{Laz13,LR16}\label{lem.You} Let $1\leq p<\infty$ and $\alpha>0$. Then we have that
\begin{align}
&\|f\ast g\|_{L^p_{uloc}(\RR^n)}\leq \|f\|_{L^1(\RR^n)}\|g\|_{L^p_{uloc}(\RR^n)},\label{eq.you1}\\
&\|f\ast g\|_{L^{\infty}(\RR^n)}\leq
C_{\alpha}\|(1+|x|)^{\frac{n+\alpha}2}f\|_{L^2(\RR^n)}\|g\|_{L^2_{uloc}(\RR^n)}.\label{eq.you2}
\end{align}
\end{lemma}
\begin{proof}[Proof of Lemma \ref{lem.G_t}]
Since \eqref{G_t-p} and \eqref{G_t-Lp'} have been proved in \cite{MYZ08}, we only need to
prove \eqref{G_t-Lp}. 

From \eqref{G_t-p}, we easily see that $\Lambda^{\alpha}G_1\in
L^1(\RR^n)$ for all $\alpha\geq 0$. Hence, by \eqref{eq.you1} in Lemma \ref{lem.You}, we
deduce that
\begin{equation*}
\begin{split}
\|\Lambda^{\alpha_1}G_t\ast f\|_{L^r_{uloc}(\RR^n)}=&\|\Lambda^{\alpha_1-\alpha_2}G_t\ast \Lambda^{\alpha_2}f\|_{L^r_{uloc}(\RR^n)}\leq
\|\Lambda^{\alpha_1-\alpha_2}G_t\|_{L^1(\RR^n)}\|\Lambda^{\alpha_2}f\|_{L^r_{uloc}(\RR^n)}\\
\leq& C_{s,\alpha}t^{-\frac{\alpha_1-\alpha_2}{2s}}\|\Lambda^{\alpha_2}f\|_{L^r_{uloc}(\RR^n)},\quad 0\leq \alpha_2\leq \alpha_1.
\end{split}
\end{equation*}
This implies \eqref{G_t-Lp}.
On the other hand, in view of \eqref{G_t-p}, we have that
\begin{align*}
\sum_{x_0\in \RR^3}\sup_{x\in B_1(x_0)}|\Lambda^{\alpha}G_1(x)|\leq C_{s,\alpha},\quad \alpha\geq 0.
\end{align*}
Hence, for all $x\in \RR^3$ and $t>0$, we get that
\begin{equation*}
\begin{split}
|[\Lambda^{\alpha_1}G_t\ast f](x)|=&t^{-\frac{\alpha_1-\alpha_2}{2s}}\Big|\int_{\RR^3}\Lambda^{\alpha_1-\alpha_2}G_1(y)
f(x-t^{1/(2s)}y)\,\mathrm{d}y\Big|\\
\leq &t^{-\frac{\alpha_1-\alpha_2}{2s}}\sum_{x_0\in \RR^3}\sup_{y\in B_1(x_0)}|\Lambda^{\alpha_1-\alpha_2}G_1(y)|\int_{y\in B_1(x_0)}|f(x-t^{1/(2s)}y)|\,\mathrm{d}y\\
\leq &C_{s,\alpha}t^{-\frac{\alpha_1-\alpha_2}{2s}}\sup_{z_0\in\RR^3}\int_{|z-z_0|\leq t^{1/2s}}|f(z)|\,\mathrm{d}y\\
\leq & C_{s,\alpha}t^{-\frac{\alpha_1-\alpha_2}{2s}}\max\{1,t^{-\frac{n}{2sr}}\}\|f\|_{L^p_{uloc}(\RR^n)}
\end{split}
\end{equation*}
which implies \eqref{G_t-Lp''}. Then we complete the proof of Lemma \ref{lem.G_t}.
\end{proof}
\begin{lemma}\label{lem.O_t}
Let $s>0$, then there exists a constant $C>0$ such that for all $x\in \RR^n$,
\begin{equation}\label{O_t-p}
|\Lambda^{\alpha}\mathcal{O}_{j,k,1}(x)|\leq C (1+|x|)^{-n-\alpha},\,\,\alpha\geq 0.
\end{equation}
In addition, if $f\in L^r(\RR^n)$ or
$L^r_{uloc}(\RR^n)$, then we have that
\begin{equation}\label{O_t-Lp'}
\|\Lambda^{\alpha_1}\mathcal{O}_{j,k,t}\ast f\|_{L^p(\RR^n)}\leq C_{\alpha_1,\alpha_2,p,r}t^{-\frac {\alpha_1-\alpha_2}{2s}-\frac
n{2s}(\frac 1r-\frac 1p)}\|\Lambda^{\alpha_2}f\|_{L^r(\RR^n)},\quad 1\leq r\leq p\leq \infty,
\end{equation}
for all $0\leq \alpha_2\leq \alpha_1$ expect $\alpha_1=\alpha_2=0$ when $r=1$ or $r=\infty$, and
\begin{align}
&\|\Lambda^{\alpha_1}\mathcal{O}_{j,k,t}\ast f\|_{L^r_{uloc}(\RR^n)}\leq C_{\alpha_1,\alpha_2}t^{-\frac{\alpha_1-\alpha_2}
{2s}}\|f\|_{L^r_{uloc}(\RR^n)}, \quad 1\leq r\leq \infty,\label{O_t-Lp}\\
&\|\Lambda^{\alpha_1}\mathcal{O}_{j,k,t}\ast f\|_{L^{\infty}(\RR^n)}\leq C_{\alpha_1,\alpha_2}t^{-\frac{\alpha_1-\alpha_2}
{2s}}\max\{1,t^{-\frac3{2pr}}\}\|f\|_{L^r_{uloc}(\RR^n)}, \quad 1\leq r<\infty\label{O_t-Lp''}
\end{align}
for all $0\leq \alpha_2<\alpha_1$
\end{lemma}
The proof is similar to the proof of Lemma \ref{lem.G_t}. Here we omit it.
\subsection{Estimates of nonlocal commutator $[\Lambda^s,\varphi]$}
In this part, we introduce a nonlocal commutator $[\Lambda^s,\varphi]$ with $s\in (0,1)$ and $\varphi\in
C^{\infty}(\RR^3)$, $\varphi\neq 0$, which will be used to deal with the nonlocal effect of $\Lambda^s$. It's firstly established in \cite{Laz13} by Lazar.
For the convenience of readers, we prove it here again.
\begin{lemma}\label{lem.com}
Let $\varphi\in C^1(\RR^3)$ be a nonnegative function satifying $\varphi\neq 0$ and $|\nabla \varphi|\neq 0$. Then for all $s\in (0,1)$, we have
\[\|[\Lambda^{s},\varphi]u\|_{L^p_{uloc}(\RR^n)}\leq
C_{s}M_{\varphi}\|u\|_{L^p_{uloc}},\quad\|[\Lambda^{s},\varphi]u\|_{L^{p,p}_{uloc}(0,T,\RR^n)}\leq
C_sM_{\varphi}\|u\|_{L^{p,p}_{uloc}(0,T,\RR^n)}\]
for any $1\leq p\leq\infty$. Here
$$M_{\varphi}\triangleq\|\varphi\|^{1-s}_{L^{\infty}(\RR^n)}\|\nabla \varphi\|^{s}_{L^{\infty}(\RR^n)}.$$
\end{lemma}
\begin{proof} It is obvious that
\begin{align*}
[\Lambda^{s},\varphi]u=&C_{s}\int_{\RR^n}{\varphi(x)u(x)-\varphi(y)u(y)\over
|x-y|^{n+s}}\,dy-\int_{\RR^n}{\varphi(x)u(x)-\varphi(x)u(y)\over |x-y|^{n+s}}\,dy\\
=&C_{s}\int_{\RR^n}{\varphi(x)u(y)-\varphi(y)u(y)\over |x-y|^{n+s}}\,dy\\
\leq&C_{s}\int_{\RR^n}\frac{\min\{\|\nabla\varphi\|_{L^{\infty}}|x-y|,
\|\varphi\|_{L^{\infty}(\RR^n)}\}}
{|x-y|^{n+s}}|u(y)|\,\mathrm{d}y\\
\triangleq& (K\ast|u|)(x)
\end{align*}
where
$K(x)=C_{s}{\min\{|x|\|\nabla\varphi\|_{L^{\infty}(\RR^n)},\|\varphi\|_{L^{\infty}(\RR^n)}\}\over
|x|^{n+s}}$. By a simple calculation, we have
$$\|K\|_{L^1(\RR^n)}\leq C_{s}\|\varphi\|^{1-s}_{L^{\infty}(\RR^n)}\|\nabla
\varphi\|^{s}_{L^{\infty}(\RR^n)}.$$
Thus, by \eqref{eq.you1} in Lemma \ref{lem.You} and Minkowski's inequality , we have that for any $1\leq p<\infty$
\[\|[\Lambda^{s},\varphi]u\|_{L^p_{uloc}(\RR^n)}\leq
C_{s}\|\varphi\|^{1-s}_{L^{\infty}(\RR^n)}\|\nabla
\varphi\|^{s}_{L^{\infty}(\RR^n)}\|u\|_{L^p_{uloc}(\RR^n)}\]
and
\[\|[\Lambda^{s},\varphi]u\|_{L^p_{uloc}(0,T,\RR^n)}\leq
C_{s}\|\varphi\|^{1-s}_{L^{\infty}(\RR^n)}\|\nabla
\varphi\|^{s}_{L^{\infty}(\RR^n)}\|u\|_{L^p_{uloc}(0,T,\RR^n)}.\]
\end{proof}
As an application of Lemma \ref{lem.com}, we have the following corollary.
\begin{corollary}\label{Cor.com}
Let $\varphi\in C^{1}_0(\RR^n)$ be a nonnegative function satisfying $\varphi\neq 0$ and $|\varphi|\neq 0$. Then for all $s\in(0,1)$, we have that
\begin{equation}\label{eq.equv}
\begin{split}
&\|[\varphi,\Lambda^s]u\|_{L^p(\RR^n)}\leq C_{s,d_{\varphi}}(M_{\phi}+1)\|u\|_{L^p_{uloc}(\RR^n)},\\
&\|[\varphi,\Lambda^s]u\|_{L^p((0,T);L^p(\RR^n))}\leq C_{s,d_{\varphi}}(M_{\phi}+1)\|u\|_{L^{p,p}_{uloc}(T)(\RR^n)}
\end{split}
\end{equation}
where $M_{\varphi}$ is defined in Lemma \ref{lem.com} and $d_{\varphi}$ is the diameter of the support of $\varphi$.
\end{corollary}
\begin{proof}
Let $\psi\in \mathcal{D}(\RR^n)$ be a nonnegative function satisfying $\psi=1$ in $B_{d_{\varphi}+1}$. Obviously, we have that
\begin{align*}
(1-\psi)[\varphi,\Lambda^s]u=&-(1-\psi)
\Lambda^s(\varphi u)=-(1-\psi)
\Big(C_s\frac{I_{|x|\geq 1}}{|x|^{n+s}}\ast u\varphi\Big).
\end{align*}
Hence, by Lemma \ref{lem.com} and Young's inequality, we deduce that for all $1\leq p\leq \infty$
\begin{align*}
\|[\varphi,\Lambda^s]u\|_{L^p(\RR^n)}\leq&\|\psi[\varphi,\Lambda^s]u\|_{L^p(\RR^n)}
+\|(1-\psi)[\varphi,\Lambda^s]u\|_{L^p(\RR^n)}\\
\leq& C_{s,d_{\varphi}}M_{\varphi}\|u\|_{L^p_{uloc}(\RR^n)}+C_s\|u\varphi\|_{L^p(\RR^n)}\\
\leq &C_{s,d_{\varphi}}(M_{\varphi}+1)\|u\|_{L^p_{uloc}(\RR^n)}.
\end{align*}
Similarly, we deduce that
\[\|[\phi,\Lambda^s]u\|_{L^{p,p}(0,T,\RR^n)}\leq C_{s,d_{\phi}}(M_{\phi}+1)\|u\|_{L^{p,p}(0,T,\RR^n)},\quad 1\leq p\leq \infty.\]
Then, we complete the proof of Corollary \ref{Cor.com}.
\end{proof}
As a consequence of Lemma \ref{lem.com} and Corollary \ref{Cor.com}, we obtain the following lemma. The proof is obvious. Here we omit it.
\begin{lemma}\label{lem.eqvi}
Let $s\in (1,2)$, then the norm $\|\cdot\|_{H^s_{uloc}}$ has the following equivalence relations:
\begin{align*}
&\|u\|_{H^s_{uloc}(\RR^n)}\simeq_{s} \|u\|_{L^2_{uloc}(\RR^n)}+\|\Lambda^su\|_{L^2_{uloc}(\RR^n)}, \quad \text{if }s\in(0,1),\\
&\|u\|_{H^1_{uloc}(\RR^n)}\simeq \|u\|_{L^2_{uloc}(\RR^n)}+\|\nabla u\|_{L^2_{uloc}(\RR^n)},\\
&\|u\|_{H^s_{uloc}(\RR^n)}\simeq_{s-1}\|u\|_{L^2_{uloc}(\RR^n)}+\|\nabla u\|_{L^2_{uloc}(\RR^n)})+\|\Lambda^{s-1}\nabla u\|_{L^2_{uloc}(\RR^n)}, \quad \text{if }s\in(1,2).
\end{align*}
\end{lemma}
\section{Local in time existence of local Leray solutions}\label{sec.3}
\setcounter{section}{3}
\setcounter{equation}{0}

In this section, we will proceed in following 5 steps to show the local-in-time existence of local Leray solutions to \eqref{NS}-\eqref{NSI} with initial data $u_0\in L^2_{uloc}$.

\textbf{Step 1. Approximate solution sequence.}
In this step, we consider the following mollified system of \eqref{NS}-\eqref{NSI}:
\begin{equation}\tag{MFNS}\label{MNS}
\left\{\begin{array}{ll}
\partial_t
u_{\varepsilon}+\Lambda^{2s}u_{\varepsilon}+J_{\varepsilon}u_{\varepsilon}\cdot\nabla
u_{\varepsilon}+\nabla P_{\varepsilon}=0,\quad\Div u_{\varepsilon}=0,\\
u_{\varepsilon}|_{t=0}=u_0,
\end{array}\right.
\end{equation}
where $J_{\varepsilon}f:=\varepsilon^{-3}\int_{\RR^3}\eta(\frac{x-y}{\varepsilon})f(y)
\,\mathrm{d}y$ with $\eta\in \mathcal{D}(B_1)$ and $0\leq \eta\leq 1$. Applying Leray
projection operator $\mathbb{P}$ to the first equation of \eqref{MNS}, by the homogeneous principle, we can rewrite \eqref{MNS} as follows:
\begin{equation}\label{INS}
\begin{split}
u_{\varepsilon}&=e^{- t\Lambda^{2s}}u_0+\int^t_0e^{-
(t-\tau)\Lambda^{2s}}\Div \mathbb{P}\big(-J_{\varepsilon}u_{\varepsilon}
\otimes u_{\varepsilon}\big)(\tau)\,\mathrm{d}\tau\triangleq e^{-t\Lambda^{2s}}u_0+B_{\varepsilon}(u_{\varepsilon},u_{\varepsilon}).
\end{split}
\end{equation}
Next, we will use Banach fixed point Theorem given in \cite{Can04} to prove the existence of solutions to \eqref{INS} in the Banach space $(X_{T_{\varepsilon}},\|\cdot\|_{X_{T_{\varepsilon}}})$ with $T_{\varepsilon}\leq 1,\,\varepsilon\leq 1$ defined as follows:
\begin{align*}
&X_{T_{\varepsilon}}\triangleq\{ u\in
L^{\infty}([0,T_{\varepsilon}];L^2_{uloc})\,\big|\,\Lambda^s u\in
L^2([0,T_{\varepsilon}];L^2_{uloc})\},\\
&\|u\|_{X_{T_{\varepsilon}}}\triangleq\|u\|_{L^{\infty}([0,T_{\varepsilon}];L^2_{uloc})}
+\|\Lambda^s u\|_{L^2([0,T_{\varepsilon}];L^2_{uloc})}.
\end{align*}

To realize it, we first show $e^{-t\Lambda^{2s}}u_0\in X_{T_{\varepsilon}}$.
By \eqref{G_t-Lp} in Lemma \ref{lem.G_t}, we have that
\[\|e^{- t\Lambda^{2s}}u_0\|_{L^{\infty}([0,T_{\varepsilon}];L^2_{uloc})}\lesssim
\|u_0\|_{L^2_{uloc}}.\]
To prove 
\begin{equation}\label{est.1}
\|\Lambda^s e^{-t\Lambda^{2s}}u_0\|_{L^{2,2}_{uloc}(T_{\varepsilon})}\lesssim_s\|u_0\|_{L^2_{uloc}},
\end{equation} 
we split $u_0$ into two parts: $u^1_0\triangleq \chi_{B_3(x_0)}u_0$ and $u^2_0\triangleq \chi_{B^c_{3}(x_0)}u_0$.
It's obvious that
\[\|\Lambda^{s}e^{-t\Lambda^{2s}}u^1_0\|_{L^2([0,T_{\varepsilon}];L^2)}\lesssim
\|u^1_0\|_{L^2}\lesssim\|u_0\|_{L^2_{uloc}}.\]
To show $\Lambda^s e^{-t\Lambda^{2s}}u^2_0\in L^{2,2}_{uloc}(T_{\varepsilon})$, we will invoke the following lemma.
\begin{lemma}\label{Lem.3-1}
Let
\[T_{x_0,R,\alpha}f=\int_{B^c_{R}(x_0)}\frac{1}{|x-y|^{3+\alpha}}f(y)\,\mathrm{d}y,\,\quad x_0\in \RR^3,\,R>2\]
Then for all $x_0\in \RR^3$, we have that
\[\|T_{x_0,R,\alpha}f\|_{L^{\infty}(B_1(x_0))}\leq C_{\alpha}R^{-s}\|f\|_{L^1_{uloc}}.\]
\end{lemma}
\begin{proof}
Since $|x-y|\geq R|y-x_0|/2$ for $x\in B_1(x_0)$, $y\in B^c_{R}(x_0)$, we have that for all $x\in B_1(x_0)$
\begin{align*}
T_{x_0,R,\alpha}f\leq& C_{\alpha} \int_{|y-x_0|\geq R}\frac{1}{|x_0-y|^{3+\alpha}}f{y}\,\mathrm{d}y\\
\leq& C_{\alpha}\sum^{\infty}_{i=1}\int_{R\cdot 2^{i-1}<|y-x_0|\leq R\cdot 2^i}\frac{
1}{(|x_0-y|)^{3+\alpha}}|f(y)|\,dy\\
\leq &C_{\alpha}\sum^{\infty}_{i=1}\frac {(R\cdot 2^{i-1})^3}{(R\cdot
2^{i-1})^{3+\alpha}}\|f\|_{L^1_{uloc}}\leq C_{\alpha} R^{-1} \|f\|_{L^1_{uloc}}.
\end{align*}
Then we complete the proof of Lemma \ref{Lem.3-1}
\end{proof}
Hence, by \eqref{G_t-p} in Lemma \ref{lem.G_t} and Lemma \ref{Lem.3-1}, we easily deduce that for all $x_0\in \RR^3$
\[\|\Lambda^se^{- t\Lambda^{2s}}u^2_0\|_{L^{\infty}([0,T_{\varepsilon}]\times B_1(x_0))}\lesssim_s \|u_0\|_{L^2_{uloc}}.\]
Then, collecting the above three estimates, we get that
\begin{equation}\label{est.2}
\|e^{-t\Lambda^{2s}}u_0\|_{X_{T_{\varepsilon}}}\leq C_1
(1+\sqrt{T_{\varepsilon}})\|u_0\|_{L^2_{uloc}}
\end{equation}
for some positive constant $C_1$ depending only on $s$. Next, we deal with the bilinear term $B_{\varepsilon}(v_1,v_2)$ with $v_1$, $v_2\in X_{T_{\varepsilon}}$. Setting $F=-J_{\varepsilon}v_1\otimes v_2$,
by \eqref{O_t-Lp} in Lemma \ref{lem.O_t}, we have that
\begin{align*}
\|B_{\varepsilon}(v_1,v_2)\|_{L^2_{uloc}}
+\|\Lambda^sB_{\varepsilon}(v_1,v_2)\|_{L^2_{uloc}}
&\lesssim \int^t_0(t-\tau)^{-\frac
1{2s}}(\|F(\tau)\|_{L^2_{uloc}}+\|\Lambda^sF(\tau)\|_{L^2_{uloc}})\,\mathrm{d}\tau.
\end{align*}
which shows that
\begin{equation}\label{est-3-1}
\|B_{\varepsilon}(v_1,v_2)\|_{X_{T_{\varepsilon}}}\lesssim_s T_{\varepsilon}^{1-1/(2s)}\|F\|_{X_{T_{\varepsilon}}}.
\end{equation}
To proceed \eqref{est-3-1}, by \eqref{eq.you2} with $\alpha=1$ in Lemma \ref{lem.You}, we observe that
\begin{align}
&\|J_{\varepsilon}v_1\|_{L^{\infty}(\RR^3\times[0,T_{\varepsilon}])}\lesssim
\varepsilon^{-3/2}\|v_1\|_{L^{\infty}([0,T_{\varepsilon}];L^2_{uloc})},\label{eq.1}\\
&\|\nabla J_{\varepsilon}v_1\|_{L^{\infty}(\RR^3\times[0,T_{\varepsilon}])}\lesssim
\varepsilon^{-5/2}\|v_1\|_{L^{\infty}([0,T_{\varepsilon}];L^2_{uloc})}.\label{eq.2}
\end{align}
With the aid of \eqref{eq.1}, we get that
\begin{align*}
\|F\|_{L^{\infty}([0,T_{\varepsilon}];L^2_{uloc})}&\leq
\|J_{\varepsilon}v_1\|_{L^{\infty}(\RR^3\times[0,T_{\varepsilon}])}
\|v_2\|_{L^{\infty}([0,T_{\varepsilon}];L^2_{uloc})}\\
&\lesssim\varepsilon^{-3/2}\|v_1\|_{L^{\infty}([0,T_{\varepsilon}];L^2_{uloc})}
\|v_2\|_{L^{\infty}([0,T_{\varepsilon}];L^2_{uloc})}.
\end{align*}
In addition, invoking Lemma \ref{lem.com}, \eqref{eq.1} and \eqref{eq.2}, we obtain that
\begin{align*}
&\|\Lambda^s F\|_{L^2([0,T_{\varepsilon}];L^2_{uloc})}\\
&\lesssim\|[\Lambda^s,J_{\varepsilon}v_1]v_2\|_{L^2([0,T_{\varepsilon}];L^2_{uloc})}
+\|J_{\varepsilon}v_1\|_{L^{\infty}(\RR^3\times[0,T_{\varepsilon}])}
\|\Lambda^sv_2\|_{L^2([0,T_{\varepsilon}];L^2_{uloc})}\\
&\lesssim (\varepsilon^{-3/2-s}\sqrt{T_{\varepsilon}}\|v_2\|_{L^{\infty}([0,T_{\varepsilon}];L^2_{uloc})}
+\varepsilon^{-3/2}\|\Lambda^sv_2\|_{L^2([0,T_{\varepsilon}];L^2_{uloc})})\|v_1\|_{L^{\infty}([0,T_{\varepsilon}];L^2_{uloc})}
.
\end{align*}
Substituting the above two estimates into \eqref{est-3-1}, we
finally obtain that
\begin{equation}\label{est.3}
\|B_{\varepsilon}(v_1,v_2)\|_{X_{T_{\varepsilon}}}\leq C_2\varepsilon^{-3}T_{\varepsilon}^{1/3}
\|v_1\|_{X_{T_{\varepsilon}}}\|v_2\|_{X_{T_{\varepsilon}}}
\end{equation}
where we use that $\varepsilon<1$, $T_{\varepsilon}<1$ and $3/4<s<1$.

Finally, collecting \eqref{est.2} and \eqref{est.3}, applying Banach fixed
Theorem to \eqref{INS}, there exists a unique mild solution to \eqref{MNS} in
$X_{T_{\varepsilon}}$ provided that
\[T_{\varepsilon}<\min\{1,\varepsilon^{9}(4C_1C_2\|u_0\|_{L^2_{uloc}})^{-3}
\}.\]

\textbf{Step 2. Uniform existence time.}
Since $u_{\varepsilon}\in L^{\infty}((0,T_{\varepsilon});L^2_{uloc})$, the equation
\[\nabla P_{\varepsilon}=-\nabla\frac{\Div\Div}{\Delta}\big(J_{\varepsilon}u_{\varepsilon} \otimes u_{\varepsilon}\big)\]
is well-defined.
Thus, $P_{\varepsilon}$ is defined up to a function which does not depend on $x$.
For each nonnegative function $\varphi\in \mathcal{D}(\RR^3)$, multiplying the first equation of
\eqref{MNS} by $u_{\varepsilon}\varphi$ and then integrating over $\RR^3\times
[0,t]$, we get that
\begin{equation}\label{eq.loc'}
\begin{split}
&\int_{\RR^3}\varphi|u_{\varepsilon}(t)|^2\,\mathrm{d}x+2\int^t_0\int_{\RR^3}
\varphi_{x_0}\big|\Lambda^{s}u_{\varepsilon}\big|^2\,\mathrm{d}x\mathrm{d}\tau\cdot\\
&=\int_{\RR^3}\varphi|u_0|^2\,\mathrm{d}x-\int^t_0\int_{\RR^3}\big([\tilde{\varphi},\Lambda^{s}]\Lambda^s
u_{\varepsilon}
\cdot u_{\varepsilon}\varphi+\tilde{\varphi}\Lambda^s
u_{\varepsilon}\cdot[\Lambda^{s},\varphi]u_{\varepsilon}\big)\,\mathrm{d}x\mathrm{d}\tau
\\
&\quad+\int^t_0\int_{\RR^3}|u_{\varepsilon}|^2
J_{\varepsilon}u_{\varepsilon}\cdot\nabla\varphi\,\mathrm{d}x\mathrm{d}\tau
-\int^t_0\int_{\RR^3}\nabla P_{\varepsilon}\cdot
u_{\varepsilon}\varphi\,\mathrm{d}x\mathrm{d}\tau
\end{split}
\end{equation}
where $\tilde{\varphi}\in \mathcal{D}(\RR^3)$ satisfies $\varphi\tilde{\varphi}=1$.
Let $\phi\in \mathcal{D}(B_2(0))$ be a nonnegative function satisfying $\varphi(x)=1$ in $B_1(0)$ and $0\leq
\phi(x)\leq 1$. Apply $\varphi=\phi_{x_0}=\phi(\cdot-x_0)$ and $\tilde{\varphi}=\phi((\cdot-x_0)/2)$ to \eqref{eq.loc'}, and then denote by $\text{I.1}$ to $\text{I.3}$ the last three terms on right-hand side of \eqref{eq.loc'}, respectively.
Set
\[\alpha_{\varepsilon}(t)=\sup_{x_0\in\RR^3}\big(\int_{B_1(x_0)}|u_{\varepsilon}(t)|^2\,dx\big)^{1/2},
\quad \beta_{\varepsilon}(t)=\sup_{x_0\in\RR^3}\big(\int^t_0\int_{B_1(x_0)}|\Lambda^s
u_{\varepsilon}|^2\,\mathrm{d}x\mathrm{d}\tau\big)^{1/2},\]
\[\gamma_{\varepsilon}(t)=\sup_{x_0\in\RR^3}\big(\int^t_0\int_{B_1(x_0)}|
u_{\varepsilon}|^3\,\mathrm{d}x\mathrm{d}\tau\big)^{1/3}.\]
Then, we claim that
\begin{equation}\label{inter}
\begin{split}
\gamma^3_{\varepsilon}(t)\lesssim_s\beta^{\frac3{2s}}_{\varepsilon}(t)\Big(1
+\int^t_0\alpha^{\frac{6(2s-1)}{4s-3}}_{\varepsilon}(\tau)\,\mathrm{d}\tau\Big)^{1-\frac3{4s}}
+\int^t_0
\alpha^3_{\varepsilon}(\tau)\,\mathrm{d}\tau.
\end{split}
\end{equation}
In fact, by the interpolation between $L^2$ and $L^{\frac6{3-2s}}$ and the Sobolev inequality: $\dot{H}^s\hookrightarrow L^{\frac6{3-2s}}$, we have that
\[\|\phi_{x_0}u_{\varepsilon}\|_{L^3(\RR^3)}\leq
\|\phi_{x_0}u_{\varepsilon}\|^{\frac{2s-1}{2s}}_{L^2(\RR^3)}
\|\Lambda^s(\phi_{x_0}u_{\varepsilon})\|^{\frac 1{2s}}_{L^2(\RR^3)}.\]
Furthermore, by Corollary \ref{Cor.com}, we observe that
\begin{align*}
\|\Lambda^s(\phi_{x_0}u_{\varepsilon})\|_{L^2}&\leq
\|[\Lambda^s,\phi_{x_0}]u_{\varepsilon}\|_{L^2}
+\|\phi_{x_0}\Lambda^s u_{\varepsilon}\|_{L^2}\leq \|\phi_{x_0}\Lambda^su_{\varepsilon}\|_{L^2}+C_s\|u_{\varepsilon}(t)\|_{L^2_{uloc}}.
\end{align*}
Collecting the above two inequalities, we easily deduce \eqref{inter}.

Next, we estimate $\text{I.1}-\text{I.3}$ in turns. By H\"{o}lder's and Young's inequalities, and Lemma \ref{lem.com}, we have that
\begin{align*}
|\text{I.1}|+|\text{I.2}|\leq
\frac12\beta^2_{\varepsilon}(t)+C_s\int^t_0\alpha^2_{\varepsilon}(\tau)\,\mathrm{d}\tau+\gamma^3_{\varepsilon}(t).
\end{align*}
To estimate $\text{I.3}$, we only need the information of $P_{\varepsilon}$ in $B_{2}(x_0)$. Hence, for any fixed $x\in B_2(x_0)\times(0,T)$, we decompose $P_{\varepsilon}(x,t)$ in the following way: take a nonnegative function $\psi_{x_0}\triangleq\psi(x-x_0)$ with $\psi\in
\mathcal{D}(B_{8}(0))$ and $\psi=1$ in $B_{4}(0)$. Then there exists a function
$P_{x_0,\varepsilon}(t)$ depending only on $x_0,t,\psi_{x_0}$ such that for $(x,t)\in
B_2(x_0)\times (0,T)$
\begin{equation}\label{eq.press}
\begin{split}
P_{\varepsilon}(x,t)=& -\Delta^{-1}\Div\Div (J_{\varepsilon}u_{\varepsilon}\otimes
u_{\varepsilon}\psi_{x_0})\\
&-\int_{\RR^3}\big(k(x-y)-k(x_0-y)\big)(J_{\varepsilon}u_{\varepsilon}\otimes
u_{\varepsilon})(y,t)(1-\psi_{x_0}(y))\,dy+P_{x_0,r,\varepsilon}(t)\\
\triangleq&
P^1_{x_0,\varepsilon}(x,t)+P^2_{x_0,\varepsilon}(x,t)+P_{x_0,\varepsilon}(t)\triangleq
P_{x_0,\varepsilon}(x,t)+P_{x_0,\varepsilon}(t).
\end{split}
\end{equation}
Then we can replace $P_{\varepsilon}$ by $P_{x_0,\varepsilon}$.
Furthermore, by Calderon-Zygmund estimate, we observe that
\[\|P^1_{x_0,\varepsilon}(t)\|_{L^{\frac32}([0,t]\times\RR^3)}\lesssim
\|J_{\varepsilon}u_{\varepsilon}\otimes u_{\varepsilon}\|_{L^{\frac32}([0,t]\times
B_8(x_0))}\lesssim \gamma^2_{\varepsilon}(t).\]
For $P^2_{x_0,\varepsilon}(x,t)$, by \eqref{eq.press1} and Lemma \ref{Lem.3-1}, we have that for all $x\in B_2(x_0)$
\begin{align*}
|P^2_{x_0,\varepsilon}(x,t)| \lesssim \alpha^2_{\varepsilon}(t).
\end{align*}
Combining with the above two estimates, we deduce that
\begin{align*}
|\text{I.3}|\lesssim  &\|P^1_{x_0,\varepsilon}\|^2_{L^{\frac32}([0,t];L^{\frac32}(\RR^3))}
\|u_{\varepsilon}\|_{L^3([0,t];L^3(B_4(x_0)))}\\
&+\int^t_0\|P^2_{x_0,\varepsilon}(\tau)\|_{L^{2}(B_{4}(x_0))}
\|u_{\varepsilon}(\tau)\|_{L^2(B_4(x_0))}\mathrm{d}\tau
\lesssim\gamma^3_{\varepsilon}(t)+\int^t_0\alpha^3_{\varepsilon}(\tau)\mathrm{d}\tau.
\end{align*}
Collecting the estimates of $\text{I.1}$-$\text{I.3}$, we obtain that
\begin{align*}
\alpha^2_{\varepsilon}(t)+2\beta^2_{\varepsilon}(\tau)\leq 2\alpha^2_{\varepsilon}(0)+\frac12\beta^2_{\varepsilon}(t)+C_s\int^t_0\big(\alpha^2_{\varepsilon}(\tau)
+\alpha^3_{\varepsilon}(\tau)\big)\,\mathrm{d}\tau+C_s\gamma^3_{\varepsilon}(t).
\end{align*}
Substituting \eqref{inter} into the above inequality and then using Young's and H\"{o}lder's inequalities, we
obtain that
\begin{equation}\label{est.6}
\alpha^2_{\varepsilon}(t)+\beta^2_{\varepsilon}(t)
\leq 2\alpha^2_{\varepsilon}(0)+C_s\int^t_0\big(\alpha^2_{\varepsilon}(\tau)
+\alpha^{\frac{6(2s-1)}{4s-3}}_{\varepsilon}(\tau)\big)\,\mathrm{d}\tau.
\end{equation}
Then, by the continuity method, we get that
\[\alpha^2_{\varepsilon}(t)+\beta^2_{\varepsilon}(t)\leq 4\|u_0\|^2_{L^2_{uloc}},\quad \forall \,\,0\leq t\leq T^*\]
with
\begin{equation}\label{est.7}
T^*=\min\{1,(4C_s)^{-1},C^{-1}_s2^{-\frac{6(2s-1)}{4s-3}}\|u_0\|^{-\frac{4s}{4s-3}
}_{L^2_{uloc}}\}.
\end{equation}

\textbf{Step 3. Existence of a weak approximate solution.}
From Step 2, we know that there exits a $M>0$ depending only on $\|u_0\|_{L^2_{uloc}}$ and $T^*$ defined in \eqref{est.7} such that
\[\|u_{\varepsilon}\|_{L^{\infty}([0,T^*];L^2_{uloc})}
+\|\Lambda^su_{\varepsilon}\|_{L^{2,2}_{uloc}(T^*)}
+\|P_{x_0,\varepsilon}\|_{L^{3/2,3/2}_{uloc}(T^*)}\leq M,\quad \forall\varepsilon\in (0,1).\]
 This, together with Corollary \ref{Cor.com}, implies that for any
nonnegative function $\phi\in \mathcal{D}(\RR^3)$, $\{\phi u_\varepsilon\}$ is bounded in
$L^{\infty}([0,T^*];L^2)\cap L^2([0,T^*];\dot{H}^s)$ and $\{\phi P_{x_0,2,\varepsilon}\}$
is bounded in $L^{\frac32}([0,T^*];L^{\frac32})$. In addition, since
\[\partial_t
u_{\varepsilon}+\Lambda^{2s}u_{\varepsilon}+\mathbb{P}\big(J_{\varepsilon}u_{\varepsilon}
\cdot\nabla u_{\varepsilon}\big)=0,\]
we conclude that $\phi\partial_t u_{\varepsilon}$ remains bounded in
$L^{\frac32}([0,T^*];H^{-\frac32})$. According to the Aubin-Lions Lemma in \cite{Se15} and
the Cantor diagonal process, we can find a subsequence of
$(u_{\varepsilon},P_{x_0,\varepsilon})$, still denoted by
$(u_{\varepsilon},P_{x_0,\varepsilon})$, and a pair $(u,P)$ with $P=P_{x_0}+P_{x_0}(t)$ defined in Remark \ref{rem.1} such that
\begin{equation}\label{eq.con}
\begin{cases}
u_{\varepsilon} \overset{*}\rightharpoonup u \text{ in } L^{\infty}([0,T^*];L^2_{uloc})
\text{ and } \Lambda^su_{\varepsilon} \rightharpoonup \Lambda^su\text{ in }
L^{2,2}_{uloc}(T^*);\\
u_{\varepsilon} \rightarrow u\text{ in } L^{3}_{uloc}(T^*);\\
P_{x_0,\varepsilon}\rightharpoonup P_{x_0} \text{ in } L^{\frac32}_{uloc}(T^*).
\end{cases}
\end{equation}
Using \eqref{eq.con}, we can easily verify that $(u,P)$ satisfies
\eqref{NS} in the sense of distribution.

\textbf{Step 4. Local energy inequality for the weak limit.}
It's obvious that $(u_{\varepsilon},P_{x_0,\varepsilon})$ satisfies the following
equality:
\begin{align*}
2\int^t_0\int_{\RR^3} \psi\big|\Lambda^{s}
u_{\varepsilon}\big|^2\,\mathrm{d}x\mathrm{d}\tau=
&\int^t_0\int_{\RR^3}\big|u_{\varepsilon}\big|^2\partial_t\psi\,\mathrm{d}x\mathrm{d}\tau
-2\int^t_0\int_{\RR^3} [\tilde{\psi},\Lambda^{s}]\Lambda^su_{\varepsilon}\cdot
u_{\varepsilon}\psi\,\mathrm{d}x\mathrm{d}\tau\\
&-2\int^t_0\int_{\RR^3}
\tilde{\psi}\Lambda^su_{\varepsilon}\cdot[\Lambda^{s},\psi]u_{\varepsilon}\,\mathrm{d}x\mathrm{d}\tau\\
&+\int^t_0\int_{\RR^3}|u_{\varepsilon}|^2
J_{\varepsilon}u_{\varepsilon}\cdot\nabla\psi\,\mathrm{d}x\mathrm{d}\tau
+\int^t_0\int_{\RR^3}P_{x_0,2,\varepsilon}\cdot
u_{\varepsilon}\nabla\psi\,\mathrm{d}x\mathrm{d}\tau
\end{align*}
for any nonnegative function $\psi\in \mathcal{D}(\RR^3\times \RR^+)$ and $\tilde{\psi}\in \mathcal{D}(\RR^3)$ with $\tilde{\psi}\psi=\psi$.
According to \eqref{eq.con}, we can easily get that
\begin{equation}\label{loc-en2}
\begin{split}
2\int^t_0\int_{\RR^3} \big|\psi\Lambda^s u\big|^2\,\mathrm{d}x\mathrm{d}\tau\leq
&\int^t_0\int_{\RR^3}\big|u\big|^2\partial_t\psi\,\mathrm{d}x\mathrm{d}\tau
-2\int^t_0\int_{\RR^3} [\tilde{\psi},\Lambda^{s}]\Lambda^s u\cdot
u\psi\,\mathrm{d}x\mathrm{d}\tau\\
&-2\int^t_0\int_{\RR^3} \tilde{\psi}\Lambda^su\cdot[\Lambda^s,\psi]u\,\mathrm{d}x\mathrm{d}\tau\\
&+\int^t_0\int_{\RR^3}|u|^2
u\cdot\nabla\psi\,\mathrm{d}x\mathrm{d}\tau
+2\int^t_0\int_{\RR^3}P\cdot u\nabla\psi\,\mathrm{d}x\mathrm{d}\tau
\end{split}
\end{equation}
for any nonnegative function $\psi\in \mathcal{D}(\RR^3\times \RR^+)$ and $\tilde{\psi}\in \mathcal{D}(\RR^3)$ with $\tilde{\psi}\psi=\psi$.

\textbf{Step 5. Strong convergence of $u(t)$ to $u_0$ in $L^2_{loc}$.}
Since $\varphi\partial_t u\in
L^{3/2}([0,T^*];H^{-3/2})$ for each nonnegative function $\varphi\in \mathcal{D}(\RR^3)$, we have that $u(t)$ is continuous in
$\mathcal{D}'(\RR^3)$ with respect to $t\in [0,T^*]$. This implies that
\begin{equation}\label{est.8}
\|u_0\|_{L^2(K)}\leq \liminf_{t\to 0+}\|u(t)\|_{L^2(K)},\quad \forall \text{ compact subset} K\subset \RR^3.
\end{equation}
On the other hand, applying \eqref{eq.con} to \eqref{eq.loc'}, we deduce that for all
$t\in [0,T^*]$
\begin{equation*}
\begin{split}
&\int_{\RR^3}\varphi|u(t)|^2\,\mathrm{d}x+2\int^t_0\int_{\RR^3}
\varphi\big|\Lambda^{s}u(\tau)\big|^2\,\mathrm{d}x\mathrm{d}\tau\\
&\leq\int_{\RR^3}\varphi|u_0|^2\,\mathrm{d}x-\int^t_0\int_{\RR^3}[\tilde{\varphi},\Lambda^{s}]
\Lambda^su\cdot u\varphi\,\mathrm{d}x\mathrm{d}\tau
-\int^t_0\int_{\RR^3}\tilde{\varphi}\Lambda^su\cdot[\Lambda^{s},\varphi]
u\,\mathrm{d}x\mathrm{d}\tau\\
&\quad+\int^t_0\int_{\RR^3}|u|^2
u\cdot\nabla\varphi\,\mathrm{d}x\mathrm{d}\tau
+2\int^t_0\int_{\RR^3}P\cdot u\nabla\varphi\,\mathrm{d}x\mathrm{d}\tau
\end{split}
\end{equation*}
for all nonnegative $\varphi,\tilde{\varphi}\in \mathcal{D}(\RR^3)$ satisfying $\varphi\tilde{\varphi}=1$. Sending $t\to 0+$, obviously, we have that
\begin{equation}\label{est.9}
\limsup_{t\to 0+}\|u(t)\|_{L^2(K)}\leq \|u_0\|_{L^2(K)},\quad \forall \text{ compact subset} K\subset \RR^3.
\end{equation}
Collecting \eqref{est.8} and \eqref{est.9}, we get that
\[\lim_{t\to 0+}\|u(t)\|_{L^2(K)}= \| u_0\|_{L^2(K)},\quad \forall \text{ compact subset} K\subset \RR^3.\]
This, together with the weak convergence of $u(t)$ to $u_0$ in $L^2_{uloc}$, gives that
\[\lim_{t\to 0+}\|u(t)-u_0\|_{L^2(K)}=0,\quad \forall \text{ compact subset} K\subset \RR^3.\]

Summing up, we prove Theorem \ref{TH1}.

\section{Regularity of local Leray solutions with initial data
in $E^2$ }\label{sec.4}
\setcounter{section}{4}
\setcounter{equation}{0}

In this section, we mainly study the additional regularity of local Leray solutions when its initial data not only belong to $L^2_{uloc}$ but also vanish at infinity.
\begin{theorem}\label{th.decay}
Suppose that $u$ be a local Leray solution to
\eqref{NS}-\eqref{NSI} on $\RR^3\times (0,T)$ starting from $u_0\in E^2$ with $\Div u_0=0$. Let $\chi\in
C^{\infty}(\RR^3)$ satisfy
\[\chi=0\quad \text{in }B_1(0),\quad\chi=1\quad\text{in } B^c_2(0),\quad 0\leq \chi\leq
1.\]
Define $\chi_{R}(x)\triangleq\chi(\frac xR)$ and $M_{T'}\triangleq\|u\|_{L^{\infty}((0,T');L^2_{uloc})}+\|u\|_{L^{2,2}_{uloc}(T')}$ with $T'<T$. Then, for each $T'<T$, there exists a constant
$C_{s,T',M_{T'}}$ such that for all $t\in (0,T')$
\begin{equation}\label{eq.decay}
\|u(t,\cdot)\chi_{R}\|_{L^2_{uloc}}+\|\Lambda^s u\chi_{R}\|_{L^{2,2}(t)}\leq
C_{s,T',M_{T'}}\big(\|u_0\chi_{R}\|_{L^2_{uloc}}+R^{-1/4}\big).
\end{equation}
\end{theorem}
\begin{proof}
First, we simplify the local energy inequality \eqref{eq.loc}. Let $\theta\in \mathcal{D}(\RR)$ be a nonnegative function satisfying $\int_{\RR} \theta\,\mathrm{d}x=1$ and $\mathop{\rm
supp}\theta\subset[-1,1]$. Define $\alpha_{\eta}(t)={1\over \eta}\int^t_{-\infty}
\theta({{\tau-t_0}\over \eta})-\theta({{\tau-t_1}\over \eta})\,\mathrm{d}\tau$ with $0<\eta<t_0<t_1$ and
$t_1+\eta<T$. It's clear that  $\alpha_{\eta}(t)\in \mathcal{D}((0,T))$. Set $\varphi\in
\mathcal{D}(\RR^3)$ to be a nonnegative function. Then, applying
$\psi(x,t)\triangleq \alpha_{\eta}(t)\varphi(x)$ to \eqref{eq.loc}, we get
that
\begin{align*}
2\int^t_0\int_{\RR^3} \alpha_{\eta}\varphi\big|\Lambda^s
u\big|^2\,\mathrm{d}x\mathrm{d}\tau\leq
       &\int^t_0\int_{\RR^3}\big|u\big|^2\partial_t\alpha_{\eta}\varphi\,\mathrm{d}x\mathrm{d}\tau
       -2\int^t_0\int_{\RR^3}\alpha_{\eta} [\tilde{\varphi},\Lambda^{s}]\Lambda^s u\cdot
       u\varphi\,\mathrm{d}x\mathrm{d}\tau\\
       &-2\int^t_0\int_{\RR^3}
       \alpha_{\eta}\tilde{\varphi}\Lambda^su\cdot [\Lambda^s,\varphi]u\,\mathrm{d}x\mathrm{d}\tau\\
       &+\int^t_0\int_{\RR^3}\alpha_{\eta}|u|^2
       u\cdot\nabla\varphi\,\mathrm{d}x\mathrm{d}\tau
       +2\int^t_0\int_{\RR^3}\alpha_{\eta}P_{x_0,2}
       u\cdot\nabla\varphi\,\mathrm{d}x\mathrm{d}\tau
\end{align*}
with $P_{x_0,2}$ dedined in \eqref{pess} with $r=2$ for all $\tilde{\varphi}\in \mathcal{D}(\RR^3)$ satisfing $\tilde{\varphi}\geq 0$ and $\tilde{\varphi}\varphi=\varphi$. In view of the definition of $\alpha_{\eta}$, we know that
\begin{equation}\label{est.17}
\partial_t\alpha_{\eta}=\tfrac1{\eta}\theta(\tfrac{t-t_0}{\eta})-\tfrac1{\eta}
\theta(\tfrac{t-t_1}{\eta}),\quad\lim_{\eta\rightarrow 0}\alpha_{\eta}(t)=\left\{
\begin{array}{ll}
0,&\quad 0\leq t<t_0\text{ or }t>t_1,\\
{1/2},&\quad t=t_0\text{ or }t=t_1,\\
1,&\quad t_0<t<t_1.
\end{array}\right.
\end{equation}
Assume that $t_0$ and $t_1$ are two Lebesgue points of the map $t\mapsto
\int_{\RR^3}|u(t)|^2\varphi\,dx$. Then, sending $\eta\to
0+$, in view of \eqref{est.17}, by Lebesgue Theorem and Lebesgue dominated converge Theorem, we obtain that
\begin{equation}\label{eq.loc3}
\begin{split}
&\int_{\RR^3}\varphi\big|u(t_1)\big|^2\,\mathrm{d}x+2\int^{t_1}_{t_0}\int_{\RR^3}
\varphi\big|\Lambda^s u\big|^2\,\mathrm{d}x\mathrm{d}\tau\\
&\leq\int_{\RR^3}\varphi\big|u(t_0)\big|^2\,\mathrm{d}x
-2\int^t_0\int_{\RR^3}[\tilde{\varphi},\Lambda^{s}]\Lambda^s u\cdot u\varphi+
\tilde{\varphi}\Lambda^su\cdot[\Lambda^s,\varphi]u\,\mathrm{d}x\mathrm{d}\tau\\
&\quad+\int^t_0\int_{\RR^3}|u|^2u\cdot\nabla\varphi\,\mathrm{d}x\mathrm{d}\tau+2\int^t_0
\int_{\RR^3}P_{x_0,2}u\cdot\nabla\varphi\,\mathrm{d}x\mathrm{d}\tau.
\end{split}
\end{equation}
Since $\lim_{t\to 0+}\|u(t)-u_0\|_{L^2_{loc}}=0$
and the map $t\rightarrow
u(\cdot,t)$ is weakly continuous, we say that \eqref{eq.loc3} holds for all $t_1\in
(0,T)$ and $t_0=0$.

Next, we will prove this theorem with the help of \eqref{eq.loc3}. According to the definition of local Leray solutions, for any $T'<T$, there exists a constant
$M_{T'}>0$, which depends only $T'$, such that
\[\|u\|_{L^{\infty}([0,T'];L^2_{uloc})}+\|\Lambda^s
u\|_{L^{2,2}_{uloc}(T')}+\|u\|_{L^{3,3}_{uloc}(T')}\leq C_{M_{T'}}.\]
Let $\phi_{x_0}$ and $\tilde{\phi}_{x_0}$ be the same functions defined in Step 2 in Section \ref{sec.3}. Applying
$\varphi\triangleq\phi_{x_0}\chi^2_{R}$ to \eqref{eq.loc3}, we get that for each $0<t\leq
T'$
\begin{align*}
&\int_{\RR^3}\phi_{x_0}|u(t)\chi_{R}|^2\,\mathrm{d}x+2\int^t_0\int_{\RR^3}
\phi_{x_0}|\chi_{R}\Lambda^{s}u|^2\,\mathrm{d}x\mathrm{d}\tau\\
&\leq\int_{\RR^3}\phi_{x_0}|u_0\chi_{R}|^2\,\mathrm{d}x-2\int^t_0\int_{\RR^3}\big([\tilde{\phi}_{x_0},
\Lambda^{s}]\Lambda^su \cdot u\phi_{x_0}\chi^2_{R}+\tilde{\phi}_{x_0}\Lambda^su
\cdot[\Lambda^{s},\phi_{x_0}\chi^2_{R}]u\big)\,\mathrm{d}x\mathrm{d}\tau\\
&\quad+\int^t_0\int_{\RR^3}|u|^2
u\cdot\nabla(\phi_{x_0}\chi_{R}^2)\,\mathrm{d}x\mathrm{d}\tau
+2\int^t_0\int_{\RR^3}P_{x_0,2} u\cdot
\nabla(\phi_{x_0}\chi_{R}^2)\,\mathrm{d}x\mathrm{d}\tau.
\end{align*}
We denote the last three parts on the right side of the above inequality by
$\text{I}$-$\text{III}$, respectively.
Set
\[\alpha_R(t)=\sup_{x_0\in\RR^3}\big(\int_{B_1(x_0)}|\chi_Ru(t)|^2\,dx\big)^{1\over
2},\quad \beta_R(t)=\sup_{x_0\in\RR^3}\big(\int^t_0\int_{B_1(x_0)}|\chi_R\Lambda^s
u|^2\,\mathrm{d}x\mathrm{d}\tau\big)^{1\over 2},\]
\[\gamma_R(t)=\sup_{x_0\in\RR^3}\big(\int^t_0\int_{B_1(x_0)}|
\chi_Ru|^3\,\mathrm{d}x\mathrm{d}\tau\big)^{1\over 3}.\]
It's obvious that $\alpha_R(t)\leq \|u(t)\|_{L^2_{uloc}}$ and $\beta_R(t)\leq \|\Lambda^su\|_{L^{2,2}(t)}$. Similar to \eqref{inter}, we have that
\begin{equation}\label{inter'}
\gamma^3_{R}(t)\lesssim_s
\beta^{\frac3{2s}}_{R}(t)\Big(\int^t_0\alpha^{\frac{6(2s-1)}{4s-3}}_{R}(\tau)
\,\mathrm{d}\tau\Big)^{1-\frac{3}{4s}}+\int^t_0\alpha^{3}_{R}(\tau)
\,\mathrm{d}\tau.
\end{equation}
as follows from Lemma \ref{lem.com} and the fact that
\begin{align*}
&\|\phi_{x_0}\chi_{R}u(t)\|_{L^3}
\leq \|\phi_{x_0}\chi_Ru\|^{1-1/(2s)}_{L^2}
\|\Lambda^s(\phi_{x_0}\chi_Ru)\|^{1/(2s)}_{L^2},\\
&\|\Lambda^s(\phi_{x_0}\chi_Ru)\|_{L^2}
\leq \|\phi_{x_0}\chi_R\Lambda^s u\|_{L^2}+\|\phi_{x_0}[\Lambda^s,\chi_R]u\|_{L^2}+\|[\Lambda^s,\phi_{x_0}]\chi_R u\|_{L^2}.
\end{align*}
Next, we deal with $\text{I}$-$\text{I}$ in turns. To estimate $\text{I}$, we observe that
\begin{align*}
\chi_R[\tilde{\varphi}_{x_0},\Lambda^{s}]
\Lambda^su&=\tilde{\varphi}_{x_0}[\chi_R,\Lambda^s]\Lambda^s u+[\tilde{\varphi}_{x_0},\Lambda^s]\chi_R\Lambda^s u+[\Lambda^s,\chi_R]\tilde{\varphi}_{x_0}\Lambda^su,\\
[\Lambda^s,\varphi_{x_0}\chi^2_R]u&=[\Lambda^s,\chi_R]\varphi_{x_0}\chi_Ru+\chi_R[\Lambda^s,
\varphi_{x_0}]\chi_R u+\chi_R\varphi_{x_0}[\Lambda^s,
\chi_R]u.
\end{align*}
Hence, by Lemma \ref{lem.com} and Young's inequality, we deduce that
\begin{equation}\label{est.10}
\begin{split}
|\text{I}|\leq&C_s\big(\beta_R(t)+R^{-s}\beta(t)\big)\Big(\int^t_0\alpha_R^2(\tau)
\,\mathrm{d}\tau\Big)^{\frac12}\\
&+C_s\beta_R(t)\Big(\int^t_0\big(\alpha_R^2(\tau)+R^{-2s}\alpha^2(\tau)\big)
\,\mathrm{d}\tau\Big)^{\frac12}\\
\leq&
C_{s,T',M_{T'}}R^{-s}+\frac14\beta^2_R(t)+C_s\int^t_0\alpha^2_R(\tau)\,\mathrm{d}\tau.
\end{split}
\end{equation}
For $\text{II}$, by H\"{o}lder's inequality, we get that
\begin{equation}\label{est.11}
\begin{split}
|\text{II}|\leq&C\|u\|_{L^3([0,t];L^3(B_2(x_0)))}\big(\|u\chi_R\|^2_{L^3([0,t];L^3(B_2(x_0)))}
+R^{-1}\|u\|^2_{L^3([0,t];L^3(B_2(x_0)))}\big)\\
\leq &C_{M_{T'}}\gamma_R^{2}(t)+C_{M_{T'}}R^{-1}.
\end{split}
\end{equation}
Finally, we consider $\text{III}$. Since $\Div u=0$, we rewrite $\text{III}=\text{III.1}+\text{III.2}$ with
\[\text{III.1}\triangleq\int^t_0\int_{\RR^3}P_{x_0,2}
\big(u\cdot\nabla\varphi_{x_0}\big)\chi_{R}^2\,\mathrm{d}x\mathrm{d}\tau,
\quad\text{III.2}\triangleq2\int^t_0\int_{\RR^3}P_{x_0,2}
\big(u\cdot\nabla\chi_{R}\big)\varphi_{x_0}\chi_{R}\,\mathrm{d}x\mathrm{d}\tau.\]
By \eqref{eq.press1}, Lemma \ref{Lem.3-1}, H\"{o}ler's inequality and Calderon-Zygmund estimates, we have that
\begin{equation}\label{est.12}
\begin{split}
|\text{III.2}|\leq& CR^{-1}\|P^1_{x_0,2}\|_{L^{\frac32}([0,t];L^{\frac32}(B_2(x_0)))}
\|u\|_{L^3([0,t];L^3(B_2(x_0)))}\\
&+R^{-1}\int^t_0\|P^2_{x_0,2}\|_{L^{2}(B_2(x_0)))}
\|u\|_{L^2(B_2(x_0))}\,\mathrm{d}\tau\\
\leq&CR^{-1}\|u\|^3_{L^{3,3}_{uloc}(t)}+tR^{-1}\|u\|^3_{L^{\infty}([0,t];L^2_{uloc})}\leq
C_{T',M_{T'}}R^{-1}.
\end{split}
\end{equation}
To estimate $\text{III.1}$, we observe that
\begin{align*}
\chi_{R}(x) P^1_{x_0,2}(x)=&-(\chi_{R}(x)-\chi_{R}(x_0)[\tfrac{\Div\Div}{\Delta}u\otimes
u\psi_{x_0}](x)\\
&[\tfrac{\Div\Div}{\Delta}u\otimes u(\chi_{R}(x_0)-\chi_{R})\psi_{x_0}](x)-[\tfrac{\Div\Div}{\Delta} u\otimes u\psi_{x_0}\chi_{R}](x),
\end{align*}
Thus, by the mean value theorem, H\"{o}lder's inequality and Calderon-Zygmund estimates, we
obtain that
\begin{equation}\label{est.13}
\|\chi_{R} P^1_{x_0,2}\|_{L^{\frac32}(B_2(x_0))}\leq CR^{-1}\|u\|^2_{L^3(B_8(x_0))}+C\|u\|_{L^3(B_8(x_0))}\|\chi_Ru\|_{L^3(B_8(x_0))}.
\end{equation}
For $\chi_{R}(x) P^2_{x_0,2}(x)$, we adapt the following decomposition:
\begin{align*}
&\chi_{R}(x) P^2_{x_0,2}(x,t)\\
&=-\int_{|y-x_0|\leq 2\sqrt{R}}\big(k(x-y)-k(x_0-y)\big)(\chi_R(x)-\chi_R(y))
(1-\psi_{x_0}(y))(u\otimes u)(y,t)\,dy\\
&\quad -\int_{|y-x_0|\leq 2\sqrt{R}}\big(k(x-y)-k(x_0-y)\big)\chi_R(y)(1-\psi_{x_0}(y))(u\otimes
u)(y,t)\,dy\\
&\quad-\chi_R(x)\int_{|y-x_0|\geq 2\sqrt{R}}\big(k(x-y)-k(x_0-y)\big)
(1-\psi_{x_0}(y))(u\otimes u)(y,t)\,dy.
\end{align*}
Hence, using the mean value theorem, we
get that for all $x\in B_2(x_0)$,
\begin{align*}
|\chi_{R}(x) P^2_{x_0,2}(x,t)|\leq &C\frac1{\sqrt{R}}\int_{4\leq |y-x_0|\leq
2\sqrt{R}}|k(x-y)-k(x_0-y|
|u(y,t)|^2\,dy\\
&\quad+\int_{4\leq |y-x_0|\leq 2\sqrt{R}}\chi_R(y)|k(x-y)-k(x_0-y)||u(y,t)|^2\,dy\\
&\quad+\int_{|y-x_0|\geq 2\sqrt{R}}|k(x-y)-k(x_0-y)||u(y)|^2\,dy.
\end{align*}
Hence, in view of \eqref{eq.press1}, by Lemma \ref{Lem.3-1}, we deduce that for all $x\in
B_2(x_0)$
\begin{equation}\label{est.14}
|\chi_{R}(x) P^2_{x_0,2}(x,t)|\leq
CR^{-1/2}\|u(t)\|^2_{L^2_{uloc}}+\alpha_R(t)\|u(t)\|_{L^2_{uloc}}.
\end{equation}
Collecting \eqref{est.13} and \eqref{est.14}, we obtain that
\begin{equation}\label{est.15}
\begin{split}
|\text{III.1}|\leq&\int^t_0\|\chi_R P^1_{x_0,2}(\tau)\|_{L^{3/2}(B_2(x_0))}\|\chi_R
u(\tau)\|_{L^3(B_2(x_0))}\,\mathrm{d}\tau\\
&+\int^t_0R^{-1/2}\|u(\tau)\|^3_{L^2_{uloc}}+\alpha_R^2(\tau)
\|u(\tau)\|_{L^2_{uloc}}\,\mathrm{d}\tau\\
\leq &C_{T',M_{T'}}R^{-1/2}+C_{M_{T'}}\gamma^2_R(t)+C_{M_{T'}}\int^t_0\alpha_R^2(\tau)
\,\mathrm{d}\tau.
\end{split}
\end{equation}
Summing \eqref{est.10}-\eqref{est.12} and \eqref{est.15}, we have that for all
$t\in[0,T']$
\begin{equation}\label{est.16}
\alpha^2_R(t)+2\beta^2_{R}(t)
\leq 2\alpha^2_R(0)+C_{s,T',M_{T'}}R^{-1/2}+2^{-1}\beta^2_R(t)+C_{s,M_{T'}}\int^t_0\alpha^2_R(\tau)\,\mathrm{d}\tau+C_{M_{T'}}\gamma_R^{2}(t).
\end{equation}
 Substituting \eqref{inter'} into \eqref{est.16} and then using H\"{older}'s inequality, we obtain that
\begin{align*}
\alpha^2_R(t)+\beta^2_{R}(t)
\leq &2\alpha^2_R(0)+C_{s,T',M_{T'}}R^{-1/2}+C_{s,M_{T'}}
\Big(\int^t_0\alpha^{\frac{6(2s-1)}{4s-3}}_{R}(\tau)
\,\mathrm{d}\tau\Big)^{\frac{4s-3}{3(2s-1)}}.
\end{align*}
This yields that
\begin{align*}
\alpha^{\frac{6(2s-1)}{4s-3}}_R(t)
\leq& 2\alpha^{\frac{6(2s-1)}{4s-3}}_R(0)+C_{s,T',M_{T'}}R^{-\frac{3(2s-1)}{2(4s-3)}}+C_{s,M_{T'}}
\int^t_0\alpha^{\frac{6(2s-1)}{4s-3}}_{R}(\tau)
\,\mathrm{d}\tau.
\end{align*}
By Gronwall's inequality, we get that, for all $t\leq T'$,
\begin{equation*}
\alpha^{\frac{6(2s-1)}{4s-3}}_R(t)\leq
\Big(2\alpha^{\frac{6(2s-1)}{4s-3}}_R(0)+C_{s,T',M_{T'}}R^{-\frac{3(2s-1)}{2(4s-3)}}\Big)e^{C_{s,M_{T'}}T'}.
\end{equation*}
This implies that \eqref{eq.decay}. Then we complete the proof of Theorem \ref{th.decay}.
\end{proof}
Next, we will use Theorem \ref{th.decay} to prove Theorem \ref{TH1-1}.
\begin{proof}[Proof of Theorem \ref{TH1-1}]
It's obvious that for all $T'<T$, $u\in L^{\infty}([0,T'];E^2)$
and $\Lambda^s u\in G^{2,2}(T')$ as follows from Theorem \ref{th.decay}. Thus, in what follows, we focus on the proof of \eqref{est-1-1}.

Let $\mathcal{Q}_{L}$ be the set of all Lebesgue points in $(0,T)$. Since $u(t)$ is continuous in $\mathcal{D}'(\RR^3)$ with respect to $t\in [0,T']$, and
$u\in L^{\infty}([0,T'];E^2)$ for any $T'<T$, we
obtain that for all $t_0\in \mathcal{Q}_{L}$
\[\|u(t_0)\|_{L^2(K)}\leq \liminf_{t_1\to t_0+}\|u(t_1)\|_{L^2(K)},\quad \forall  K\subset\RR^3 \]
where $K$ denotes the compact set here and in what follows. On the other hand, the inequality \eqref{eq.loc3} tells us that for all $t_0\in \mathcal{Q}_{L}$
\[\limsup_{t_1\to t_0+}\|u(t_1)\|_{L^2(K)}\leq \|u(t_0)\|_{L^2(K)},\quad \forall K\subset \RR^3.\]
Thus, we get that for all $t_0\in \mathcal{Q}_{L}$
\[\lim_{t_1\to t_0+}\|u(t_1)\|_{L^2(K)}= \|u(t_0)\|_{L^2(K)},\quad \forall\, K\subset \RR^3.\]
This, together with the weak continuity of the map $t\rightarrow u(\cdot,t)$ and
the fact that
\[\lim_{t_1\to 0+}\|u(t_1)-u(0)\|_{L^2(K)}=0,\quad \forall\, K\subset \RR^3,\]
we obatin that for all $t_0\in \mathcal{Q}_{L}\cup\{0\}$
\begin{equation}\label{est.19}
\lim_{t_1\to t_0+}\|u(t_1)-u(t_0)\|_{L^2(K)}=0,\quad \forall\, K\subset \RR^3.
\end{equation}
In addition, since $u\in L^{\infty}([0,T'];E^2)$ which implies that
\begin{equation}\label{est.20}
\lim_{|x_0|\to\infty}\|u\|_{L^{\infty}([0,T'];L^2(B_1(x_0)))}=0,\quad \forall T'<T,
\end{equation}
collecting \eqref{est.19} and \eqref{est.20}, we deduce \eqref{est-1-1}. So we complete the proof of Theorem \ref{TH1-1}.
\end{proof}
In the light of Theorem \ref{TH1-1}, for each local Leray solution $u$ to \eqref{NS}-\eqref{NSI} with initial data in $E^2$, we can define its harmonic extension $u^*$:
\begin{equation}\label{HE1}
u^*(x,y)=C_{s}\int_{\RR^3}\frac{y^{1-2s}}{(|x-\xi|^2+|y|^2)^{\frac{3+2s}2}}u(\xi)
\,\mathrm{d}\xi.
\end{equation}
According to \cite{CS07}, such $u^*$ satisfies
\begin{equation}\label{HE}
\begin{cases}
\text{Div}(y^{1-2s}\overline{\nabla}u^*)=0,\quad(x,y)\in \RR^3\times\RR^+,\\
u^*(x,0)=u(x),\quad x\in \RR^3,\\
-C_s\lim_{y\to 0+}y^{1-2s}\partial_y u^*=(-\Delta)^{s}u(x),\quad x\in\RR^3.
\end{cases}
\end{equation}
Hence, the local energy inequality \eqref{eq.loc} for $u\in L^{\infty}((0,T);E^2)$ and $\Lambda^s u\in G^{2,2}(T)$ is equivalent to the following inequality
\begin{equation}\label{eq.loc1}
\begin{split}
&\int_{\RR^3\times\{t\}}|u|^2\psi\,\mathrm{d}x+2C_s\int^t_0\int_{\RR^3\times\RR_+}y^{1-2s}
|\overline{\nabla}u^*|^2\psi\,\mathrm{d}\bar{x}\mathrm{d}\tau\\
&\leq C_s\int^t_0\int_{\RR^3\times\RR_+}|u^*|^2\mathrm{Div}(y^{1-2s}\overline{\nabla}\psi)\,
\mathrm{d}\bar{x}\mathrm{d}\tau+\int^t_0\int_{\RR^3}(u\cdot\nabla \psi)(2P+|u|^2)\mathrm{d}x\mathrm{d}\tau\\
&\quad+\int^t_0\int_{\RR^3}|u|^2\big(\partial_t\psi+C_s\lim_{y\to 0+}(y^{1-2s}\partial_y\psi)\big)\mathrm{d}x\mathrm{d}\tau
\end{split}
 \end{equation}
for any nonnegative function $\psi\in \mathcal{D}(\RR^3\times (0,T))$. Note that, the $\psi$ of the integration on $\RR^3\times\RR^+$ in \eqref{eq.loc1} should be understood as any function in $H^1(\RR^3\times \RR^+,y^{1-2s})$ which equals $\varphi$ on $\RR^3\times\{t=0\}$.

Based on above analysis, we can give another version of suitable weak solution which is a generalization of the version introduced in \cite{TY15}.
\begin{definition}\label{def.suit}
Let $T>0$ and $s\in [\frac34,1)$. We call $(u,P)$ is a suitable weak solution to \eqref{NS} in $\RR^3\times (0,T)$, if the following conditions are satisfied:
\begin{enumerate}
  \item[$(\mathrm{i})$] $u\in L^{\infty}((0,T);E^2)$, $\Lambda^su\in G^2(0,T)$ and $P\in L^{3/ 2}((0,T);E^{3/2});$\smallskip

  \item[$(\mathrm{ii})$] $u$ and $P$ satisfy \eqref{NS} in the sense of distribution on $\RR^3\times (0,T)$;\smallskip

  \item[$(\mathrm{iii})$] $u$ and $P$ satisfy the generalized local energy inequality \eqref{eq.loc} on $(0,T)$ for any nonnegative function $\psi\in \mathcal{D}(\RR^3\times (0,T))$.
\end{enumerate}
\end{definition}
\begin{remark}
In the light of Theorem \ref{TH1-1} and Definition \ref{def.suit}, we say that the local Leray solution with initial data in $E^2$ is a suitable weak solution.
\end{remark}
\begin{proposition}\label{pro.CKN}
Let $s\in [3/4,1)$. If $(u,P)$ is a suitable weak solution to \eqref{NS} on $\RR^3\times (0,T)$, then there exists a positive small constant $\varepsilon_0$, only depending on $s$, such that if
\begin{equation}\label{cond.1}
r^{4s-6}\int_{Q_r(x_0,t_0)}|u|^3+|P|^{\frac32}\,\mathrm{d}x
\mathrm{d}\tau<\varepsilon_0,
\end{equation}
then there exists a positive constant $C$ such that
\[\sup_{z\in Q_{\frac{r}2(x_0,t_0)}}| u|\leq Cr^{-1}.\]
\end{proposition}
\begin{remark}
In the range $s\in (3/4,1)$, Proposition \ref{pro.CKN} was proved in \cite{TY15} under the following condition replacing \eqref{cond.1},
\begin{equation}\label{cond.2}
\limsup_{r\to 0+}r^{-5+4s}\int_{Q^*_r(x_0,t_0)}y^{1-2s}|\overline{\nabla}u^*|^2\,\mathrm{d}\bar{x}
\mathrm{d}\tau<\varepsilon_1.
\end{equation}
Here $Q^*_r(x_0,t_0)=B_r(x_0)\times(0,r)\times(t_0- r^{2s},t_0)$ and $\varepsilon_1$ is the constant in Theorem 1.2 in \cite{TY15}. Later, Ren, Wang and Wu in \cite{RWW16} proved Proposition \ref{pro.CKN} at the end point case $s=3/4$. Following the method used in \cite{RWW16}, we can prove Proposition \ref{pro.CKN} for all $s\in [3/4,1)$. Since this works without any additional tricks, we feel that it is
not necessary to reproduce this purely mechanical procedure here in detail.
\end{remark}
Based on Proposition \ref{pro.CKN}, we prove the local Leray solutions to \eqref{NS}-\eqref{NSI} with initial data in $E^2$ satisfy the following additional regularity.
\begin{proposition}\label{pro.regu}
Let $u$ be a local Leray solution to \eqref{NS} on $\RR^3\times (0,T)$ starting from
$u_0\in E^2$ with $\Div u_0=0$. Then $u$ satisfies the following regularity
results:
\begin{enumerate}
  \item $u$ is regular in $B^c_R$ away from $t=0$ for some large enough $R>0$;
  \item $\Lambda^s u\in L^{2}([\delta,t];E^2)$ for all $0<\delta<t<T$.
\end{enumerate}
\end{proposition}
\begin{proof}
By Theorem
\ref{th.decay}, we know that $u\in L^{\infty}((0,t);E^2)$ and $\Lambda^s u\in G^{2,2}(t)$ for any $t<T$. Hence, by the interpolation, there exists a $R>0$, such that for all
$|x_0|\geq R$ and $t_0\in (0,t]$
\[\int_{Q_r(x_0,t_0)}|u|^3+|P-P_{x_0}|^{\frac32}\,\mathrm{d}x\mathrm{d}\tau<r^{6-4s}\varepsilon_0\]
where $r^2=\min\{(T-t_0)/2,t_0/2\}$ and $\varepsilon_0$ is the constant given in Proposition
\ref{pro.CKN}. Using Proposition \ref{pro.CKN}, we get that $\|u\|_{L^{\infty}(B^c_{R}\times [\delta,t])}\leq C$ for all $0<\delta<t$. Then the first part of Proposition \ref{pro.regu} is proved.

Next, we prove the second part of this proposition. Since $u\in L^{\infty}((0,t);E^2)$ and $\Lambda^s u\in G^{2,2}(t)$ for any $t<T$, we have that for almost every $t_0\in (0,T)$
 \[\lim_{|x_0|\to \infty}\|\Lambda^s u(t_0)\|_{L^2(B_1(x_0))}=0.\]
Thus, in the definition of $E^2$, it's sufficient to prove that the \textbf{claim} that $\Lambda^s u\in L^2([\delta,t];L^2_{uloc})$. Hence, in what follows, we will prove the \textbf{claim}. 

First, invoking the fact that $\|u\|_{L^{\infty}(B^c_{R}\times [\delta,t])}\leq C$ for all $0<\delta<t$,  with the help of  nonhomogeneous Besov space $B^{s}_{p,q}$ in $\RR^3$ defined in \cite{BCD11,MWZ12}, by bootstrap, we will prove a improved regularity of $u$. In the light of Theorem \ref{TH6-1}, we have that for all $0<t_0\leq t<T$
\begin{align*}
u(t,x)=e^{-(t-t_0)\Lambda^{2s}}u(t_0)+\int^t_{t_0}\Div e^{-(t-\tau)\Lambda^{2s}}
\mathbb{P}(u\otimes u)(\tau)\,\mathrm{d}\tau.
\end{align*}
Let $\varphi_r\in\mathcal{D}(\RR^3)\,(r\geq 1)$ satisfy $0\leq \varphi_r\leq 1$, $\varphi_r=1$ in $B_{3r/2}$ and $\varphi_r=1$ in $B^c_{2r}$. We split $u$ into $u_1+u_2$ where
\[u_i(t,x)=e^{-(t-t_0)\Lambda^{2s}}\phi_i u(t_0)+\int^t_{t_0}\Div e^{-(t-\tau)\Lambda^{2s}}
\mathbb{P}(\phi_i u\otimes u)(\tau)\,\mathrm{d}\tau,\,\,\,\, \phi_1=1-\phi_2,\,\phi_2=(1-\varphi_r)^2.\]
Then, by Lemma \ref{lem.G_t} and Lemma \ref{lem.O_t}, we have that for any $|x|\geq 3r$ and $0<t_0\leq t<T$,
\begin{align*}
|u_1(t,x)|&\lesssim_s\int_{|y|\leq 2r}\frac{t}{|x-y|^{3+2s}}|u(t_0,y)|\,\mathrm{d}y+\int^t_{t_0}
\int_{|y|\leq 2r}\frac{1}{|x-y|^{4}}|u(\tau,y)|^2\,\mathrm{d}y\\
&\lesssim_s \|u(t_0)\|_{L^2(B_{2r})}+(t-t_0)\|u\|^2_{L^{\infty}([t_0,t];L^2(B_{2r}))}.
\end{align*}
Similarly, we get that
\begin{align*}
|\nabla u_1(t,x)|\leq \|u(t_0)\|_{L^2(B_{2r})}+(t-t_0)\|u\|^2_{L^{\infty}([t_0,t];L^2(B_{2r}))},\quad x\in B^c_{3r},\,0<t_0\leq t<T.
\end{align*}
Combining the above two estimates, we deduce that
\begin{equation}\label{0}\|(1-\tilde{\varphi})u_1\|_{B^{1}_{\infty,\infty}}\leq \|(1-\tilde{\varphi})u_1\|_{W^{1,\infty}}\leq \|u_1\|_{W^{1,\infty}(B^c_{3r})}<\infty
\end{equation}
where $\tilde{\varphi}_r(x)=\varphi_r(x/2)$.
To deal with $u_2$, we observe that
\begin{align*}
&\|(1-\varphi)^2u\otimes u\|_{B^{0}_{\infty,\infty}}\leq \|(1-\varphi)u\|^2_{L^{\infty}},\quad\|(1-\varphi)^2u\otimes u\|_{B^{1/2}_{\infty,\infty}}\leq \|(1-\varphi)u\|^2_{B^{1/2}_{\infty,\infty}}.
\end{align*}
Hence, by lemma 2.4 in \cite{BCD11} , we get that
\begin{align}
&\|u_2\|_{B^{1/2}_{\infty,\infty}}\lesssim_{s,t-t_0}\|(1-\varphi)u(t_0)\|_{L^{\infty}}+ \|(1-\varphi)u\|^2_{L^{\infty}([t_0,t]\times \RR^3)},\label{1'}\\
&\|u_2\|_{B^{1}_{\infty,\infty}}\leq  \lesssim_{s,t-t_0}\|(1-\varphi)u(t_0)\|_{L^{\infty}}+ \|(1-\varphi)u\|^2_{L^{\infty}([t_0,t];B^{1/2}_{\infty,\infty})}.\label{1}
\end{align}
which implies that
\begin{equation}\label{2}
\|(1-\tilde{\varphi})u_2\|_{B^{(k+1)/2}_{\infty,\infty}}\leq \|1-\tilde{\varphi}\|_{L^{\infty}}\|u_2\|_{B^{k/2}_{\infty,\infty}}
+\|u_2\|_{L^{\infty}}\|1-\tilde{\varphi}\|_{B^{k/2}_{\infty,\infty}}<\infty.
\end{equation}
Substituting \eqref{0} and \eqref{2} with $k=0$ and $r=R$ into \eqref{1} with $r=2R$, and then combining with \eqref{0} with $r=2R$, we finally deduce that $u_1\in L^{\infty}([\delta,t];W^{1,\infty}(B^c_{6R}))$ and $u_2\in L^{\infty}([\delta,t];B^{1,\infty})$.

Then, we conclude that for all $0<\delta<t<T$
\begin{align*}
\|\Lambda^s u_1\|_{L^{2}([\delta,t];L^{2}_{uloc})}&\leq \|\varphi_{4R}u_1\|_{L^{2}([\delta,t];H^{s}_{uloc})}+\|(1-\varphi_{4R})u_1\|_{L^{2}
([\delta,t];H^{s}_{uloc})}\\
&\lesssim \|u_1\|_{L^{2}([\delta,t];L^{2}_{uloc})}+\|\Lambda^s u_1\|_{L^{2}([\delta,t];L^{2}(B_{8R}))}+\| (1-\varphi_r)u_1\|_{L^{2}([\delta,t];H^{1}_{uloc})}\\
&\lesssim_{\delta,t} \|u_1\|_{L^{\infty}([\delta,t];L^{2}_{uloc})}+\|\Lambda^s u_1\|_{L^{2}([\delta,t];L^{2}(B_{8R}))}+\|u_1\|_{L^{\infty}([\delta,t];
W^{1,\infty}(B^c_{6R}))}<\infty
\end{align*}
and
\[\|\Lambda^s u_2\|_{L^{2}([\delta,t];L^{2}_{uloc})}\lesssim_{\delta,t}\|\Lambda^s u_2\|_{L^{\infty}([\delta,t];L^{\infty})}\lesssim_{\delta,t}\| u_2\|_{L^{\infty}([\delta,t];B^{1,\infty})}<\infty.\]
Then, we complete the proof of Proposition \ref{pro.regu}.
\end{proof}

\section{Global-in-time existence of local Leray solutions}\label{sec.5}
\setcounter{section}{5}
\setcounter{equation}{0}

In this section, we will adapt the construction method to prove the global in time existence of local Leray solutions to \eqref{NS}-\eqref{NSI} with initial data in $E^2$. To do it, we need the following ``weak-strong'' uniqueness and decomposition lemma.
\begin{proposition}[Weak-strong uniqueness]\label{pro.uni}
Let $(u,P),(v,\bar{P})$ be two local Leray solutions to \eqref{NS} on $\RR^3\times (0,T)$ with same initial data $u_0$. Suppose that $\nabla v\in L^{q}([0,T'];E^p)$ for all $T'<T$ with $\tfrac{2s}q+\tfrac3p=2s$ and $p\geq 2$, then $u=v$ in $\RR^3\times (0,T)$ almost everywhere.
\end{proposition}
\begin{remark}
When $s<1$, the properties of $u$ can't provide any information of $\nabla u$. Hence, unlike the case $s=1$, to control $\int^t_0\big((u-v)\cdot\nabla) (u-v)\cdot v\varphi\,\mathrm{d}\mathrm{d}\tau$,
we need at least the addition information of $\nabla v$.
\end{remark}
\begin{proof}
Since $L^{\infty}([0,T'];E^2)\cap L^2([0,T'];\bar{H}^s_{uloc})\hookrightarrow L^{\alpha}([0,T'];\bar{L}^{\beta}_{uloc})$ with $\tfrac{2s}{\alpha}+\tfrac{3}{\beta}=\frac32$, we have that
\begin{align*}
&\int_{\RR^3} v(t)\cdot u(t)\varphi\,\mathrm{d}x+2\int^t_0\int_{\RR^3}(\Lambda^s
v\cdot\Lambda^su)\varphi\,\mathrm{d}x\mathrm{d}\tau\\
&=\int_{\RR^3}|u_0|^2\varphi\,\mathrm{d}x-\int^t_0\int_{\RR^3}\big([\tilde{\varphi},\Lambda^s]\Lambda^sv\cdot
u\varphi+\varphi\Lambda^sv\cdot[\Lambda^s,\tilde{\varphi}]u\big)\,\mathrm{d}x\mathrm{d}\tau\\
&\quad-\int^t_0\int_{\RR^3}\big([\tilde{\varphi},\Lambda^s]\Lambda^su\cdot
v\varphi+\varphi\Lambda^su
\cdot[\Lambda^s,\tilde{\varphi}]v\big)\,\mathrm{d}x\mathrm{d}\tau\\
&\quad-\int^t_0\int_{\RR^3} \big((u\cdot\nabla) u\cdot v+(v\cdot\nabla) v\cdot
u\big)\varphi\,\mathrm{d}x\mathrm{d}\tau+\int^t_0\int_{\RR^3}\big((\bar{P}u+Pv)\big)
\cdot\nabla\varphi \,\mathrm{d}x\mathrm{d}\tau
\end{align*}
for any nonnegative functions $\varphi,\tilde{\varphi}\in \mathcal{D}(\RR^3)$ satisfying $\tilde{\varphi}\varphi=1$. Since $(u,P)$ and $(v,\bar{P})$ both satisfy the local energy
inequality \eqref{eq.loc}, setting $w=u-v$, from the above relation, we deduce that
\begin{align*}
&\int_{\RR^3}\varphi|w|^2\,\mathrm{d}x+2\int^t_0\int_{\RR^3}\varphi|\Lambda^s
w|^2\,\mathrm{d}x\mathrm{d}\tau\\
&\leq -\int^t_0\int_{\RR^3}\big([\tilde{\varphi},\Lambda^s]\Lambda^sw\cdot
w\varphi+\varphi\Lambda^sw\cdot[\Lambda^s,\tilde{\varphi}]w\big)\,\mathrm{d}x\mathrm{d}\tau\\
&\quad+\int^t_0\int_{\RR^3}|w|^2(w+v)\nabla
\varphi\,dx\,d\tau-2\int^t_0\int_{\RR^3}(w\cdot\nabla v)\cdot w\varphi\,\mathrm{d}x\mathrm{d}\tau\\
&\quad+\int^t_0\int_{\RR^3}(P-\bar{P})w\cdot\nabla\varphi\,\mathrm{d}x\mathrm{d}\tau
\end{align*}
for any nonnegative functions $\varphi,\tilde{\varphi}\in \mathcal{D}(\RR^3)$ satisfying $\tilde{\varphi}\varphi=1$.
Set
\[\alpha(t)=\sup_{x_0\in\RR^3}\big(\int_{B_1(x_0)}|w(t)|^2\,dx\big)^{1\over 2},\quad
\beta(t)=\sup_{x_0\in\RR^3}\big(\int^t_0\int_{B_1(x_0)}|\Lambda^s
w|^2\,\mathrm{d}x\mathrm{d}\tau\big)^{1\over 2}.\]
Then, along the exact same lines as the proof of \eqref{est.6}, we deduce that
\begin{align*}
\alpha^2(t)+\beta^2(t)\leq& M\int^t_0\big(\alpha^2(\tau)
+\alpha^{\frac{6(2s-1)}{4s-3}}(\tau)\big)\,\mathrm{d}\tau
\end{align*}
where $M$ depends only on $T$, $\|u\|_{L^{\infty}((0,T);E^2)}$, $\|\Lambda^su\|_{L^{2}((0,T);E^2)}$ and $\|\nabla u(t)\|_{L^q((0,T);E^p)}$.
This gives $\alpha(t)=0$ on $(0,T)$, which implies that $u=v$ on $(0,T)$. Then we complete the proof of Proposition \ref{pro.uni}.
\end{proof}
\begin{lemma}\label{p}
Let $s>0$ and $f\in \bar{H}^s_{uloc}(\RR^n)$. Then for all $\varepsilon>0$, there exists $g\in V^s$ and $h\in L^2(\RR^n)$ such that $f=g+h$ and $\|g\|_{\bar{H}^s_{uloc}(\RR^n)}\leq \varepsilon.$
\end{lemma}
\begin{remark}
Since the proof is completely parallel to Proposition 12.1 in \cite{LR02}, we here omit it.
\end{remark}
Next, we come back to the proof of Theorem \ref{TH2}.
\begin{proof}[Proof of Theorem \ref{TH2}]
In what follows, we will prove the global existence of local Leray solutions with initial values in $E^2$ by construction.

\textbf{Step 1: local existence in $E^2$.} For each divergence free vector $u_0\in E^2$,  Theorem \ref{TH1} tells us that there exists a local Leray
solution $u^1$ to \eqref{NS} on $\RR^3\times (0,T_1)$ stating from $u_0$.

\textbf{Step 2: Construction of local Leray solution from $(0,T_N)$ to $(0,T_{N+1})$ with $T_{N+1}-T_{N}>3/4$.} Assume that $u^N$ is a local Leray solution on $(0,T_N)$, Proposition \ref{pro.regu} tells us that there exists $S_{N+1}\in
(\sup\{0,T_N-1/4\},T_N)$ such that
\[u^N(S_{N+1})\in \bar{H}^{\frac52-2s}_{uloc}\quad\text{and}\quad\lim_{t\to S_{N+1}+}\|u^N(t)-u^N(S_{N+1})\|_{E^2}=0.\]

Next, by Theorem \ref{well-loc}, we can find a solution $X^{N+1}\in C([S_{N+1},U_{N+1}];\bar{H}^{\frac52-2s}_{uloc})$ with $S_{N+1}<U_{N+1}<T_N$ and
$X^{N+1}(S_{N+1})=u^N(S_{N+1})$.
In addition, it satisfies that
\[X^{N+1}\in L^{\frac{4s}{4s-3}}([S_{N+1},U_{N+1}]\bar{H}^{1}_{uloc})\cap L^{2}([S_{N+1},U_{N+1}];\bar{H}^{\frac52-s}_{uloc}).\]
By Corollary \ref{Cor-6.3}, $X^{N+1}$ is a local Leray solution to \eqref{NS} with initial data $u^N(S_{N+1})$. On the other hand, by Proposition \ref{p}, we can split $u^N(S_N+1)$ into
$v^{N+1}_0+w^{N+1}_0$ satisfying $\Div v^{N+1}_0=\Div w^{N+1}_0=0$,
$\|v^{N+1}_0\|_{\bar{H}^{\frac52-2s}_{uloc}}<\varepsilon$ and $w^{N+1}_0\in L^2$, where
$\varepsilon$ is chosen small enough such that \eqref{NS} admits a unique solution $v^{N+1}\in
C([S_{N+1},S_{N+1}+1];V^{\frac52-2s})$ with $v^{N+1}(S_{N+1})=v^{N+1}_0$ satisfying
\[v^{N+1}\in L^{2}([S_{N+1},S_{N+1}+1];\bar{H}^{\frac52-2s}_{uloc}).\]
Then, by Theorem \ref{TH.pertu}, there exists a weak solution $w^{N+1}\in L^{\infty}([S_{N+1},S_{N+1}+1];L^2)\cap L^2([S_{N+1},S_{N+1}+1];\dot{H}^s)$ satisfying \eqref{6.1} to the following perturbed system:
\begin{equation*}
\left\{\begin{array}{ll}
\partial_t w^{N+1}+\Lambda^{2s} w^{N+1}+w^{N+1}\cdot\nabla w^{N+1}+v^{N+1}\cdot\nabla w^{N+1}+w^{N+1}\cdot\nabla
v^{N+1}+\nabla \Pi=0,\\
\Div w^{N+1}=0,\\
w^{N+1}(S_{N+1},x)=w^{N+1}_0.
\end{array}.\right.
\end{equation*}
Set $Y^{N+1}=v^{N+1}+w^{N+1}$. In the light of Corollary \ref{Cor-6.3}, it's obvious that $Y^{N+1}$ is a local Leray solution to \eqref{NS} on $(S_{N+1},S_{N+1}+1)$ starting from $u^{N}(S_{N+1})$.

Finally, since $X^{N+1}\in L^{\frac{4s}{4s-3}}([S_{N+1},U_{N+1}];\bar{H}^1_{uloc})$, by Proposition \ref{pro.uni}, we conclude that $u^{N}=X^{N+1}=Y^{N+1}$ on $(S_{N+1},U_{N+1})$. Hence, defining $u^{N+1}=\chi_{[0,S_{N+1}]}u^{N}+Y^{N+1}$ and $T_{N+1}=S_{N+1}+1$, we construct a local Leray solution to \eqref{NS} on $(0,T_{N+1})$ with $T_{N+1}-T_N>3/4$.

\textbf{Step.3 global existence in $E^2$.}
By induction on $N$, we can construct a sequence of Leray solutions $u^N$ on $(0,T_N)$
satisfying $u^N=u^{N+1}$ on $(0,T_{N})$ and $T_{N+1}-T_{N}>3/4$. Hence, $u=\lim_{N\to\infty} u^N$ is well defined on $(0,\infty)$. Furthermore, $u$ is a local Leray solution to \eqref{NS} starting from $u_0$ on $\RR^3\times (0,\infty)$. Then we complete the proof of Theorem \ref{TH2}.
\end{proof}

\section{Appendix}\label{sec.A}
\setcounter{section}{6}
\setcounter{equation}{0}

\subsection{Equivalence of differential and integral formulations for FNS}

\begin{theorem}\label{TH6-1}
Let $u\in L^2([0,T'];L^2_{uloc})$ for any $T'<T$. Then the following two arguments are equivalent:
\begin{enumerate}
  \item [$\mathrm{(1)}$]$u$ is a solution of the differential fractional Navier-Stokes equations
  \begin{equation*}
  \begin{cases}
  \partial_t u+\Lambda^{2s}u+\mathcal{P}\Div(u\otimes u),\\
  \Div u=0.
  \end{cases}
  \end{equation*}
  \item [$\mathrm{(2)}$] $u$ is a solution of the following integral fractional Navier-Stokes equations
  \begin{equation*}
  \exists u_0\in \mathcal{S}'(\RR^n),
  \begin{cases}
  u=e^{-t\Lambda^{2s}}u_0-\int^t_0 e^{-(t-s)\Lambda^{2s}}\mathbb{P}\Div(u\otimes u)\,\mathrm{d}s,\\
  \Div u=0.
  \end{cases}
  \end{equation*}
  \end{enumerate}
In particular, if $u\in L^2([0,T'];E^2)$ for all $T'<T$, the above arguments are equivalent to that $u$ is a weak solution to \eqref{NS}.
\end{theorem}
Since the proof is completely parallel to classical Navier-Stokes system, here we omit it. For details, see Theorem 11.1 and Theorem 11.2 in \cite{LR02}.
\subsection{Kato's theory of FNS in the spaces of local
measures}

Set
\[X_T:=\{f\in L^4([0,T];\bar{H}^{(5-3s)/2}_{uloc})\cap L^{4s/(4s-3)}([0,T];\bar{H}^{1}_{uloc})\,\big|\,t^{1-1/(2s)}f\in C([0,T];L^{\infty})\}\]
with
\[\|f\|_{X_T}:=\|f\|_{L^4([0,T];\bar{H}^{(5-3s)/2}_{uloc})}
+\|f\|_{L^{4s/(4s-3)}([0,T];\bar{H}^{1}_{uloc})}+\sup_{0\leq t\leq T}t^{1-1/(2s)}\|f\|_{L^{\infty}}.\]
Then, we can prove the uniqueness existence of \eqref{NS} in $X_T$. 
\begin{theorem}\label{well-loc}
Let $s\in[5/6,1)$. Then for all divergence free vector field $u_0\in \bar{H}^{5/2-2s}_{uloc}$, there exists a $0<T\leq 1$ satisfying
\[4C_{\nu}\|e^{\nu t\Lambda^{2s}}u_0\|_{X_T}<1\]
for some constant $C$ and a unique mild solution
$$u\in C([0,T];\bar{H}^{5/2-2s}_{uloc})\cap X_T\cap L^2([0,T];\bar{H}^{5/2-s}_{uloc})$$
to \eqref{NS} with $u(\cdot,0)=u_0$. In particular, there exists a constant $\eta>0$ such that if $\|u_0\|_{\bar{H}^{5/2-2s}_{uloc}}<\eta$, then we have that $T=1$.
\end{theorem}
To prove Theorem \ref{well-loc}, we will invoke the following lemma.
\begin{lemma}\label{lem.H}
Let $s\in [5/6,1]$. Then we have that
\begin{equation}\label{est.H}
\|fg\|_{\bar{H}^{7/2-3\alpha}_{uloc}}\leq C\|f\|_{\bar{H}^{(5-3\alpha)/2}_{uloc}}\|g\|_{\bar{H}^{(5-3\alpha)/2}_{uloc}}.
\end{equation}
\end{lemma}
\begin{proof}
It's well known that 
\[\|fg\|_{H^{7/2-3\alpha}}\leq C\|f\|_{H^{(5-3\alpha)/2}}\|g\|_{H^{(5-3\alpha)/2}}.\]
Let $\phi\in \mathcal{D}(B_2)$ be a nonnegative function satisfying $\phi=1$ in $B_1$. Hence for all $\phi_{x_0}(\cdot)\triangleq\phi(\cdot-x_0)$, we have that
\[\|\phi^2_{x_0}fg\|_{H^{7/2-3\alpha}}\leq C\|\phi_{x_0}f\|_{H^{(5-3\alpha)/2}}\|\phi_{x_0}g\|_{H^{(5-3\alpha)/2}}.\]
Then, in the light of the definition of Sobolev spaces in local measure, we deduce that \eqref{est.H}.
\end{proof}
Now, we come back to the proof of Theorem \ref{well-loc}.
\begin{proof}
By homogeneous principle, we have that
\begin{align*}
u=e^{- t\Lambda^{2s}}u_0+\int^t_0\nabla\cdot e^{-(t-\tau)\Lambda^{2s}}\mathbb{P}\big(-u
\otimes u\big)(\tau)\,\mathrm{d}\tau\triangleq e^{-\nu t\Lambda^{2s}}u_0+B(u,u).
\end{align*}
Next, we will complete the proof of this theorem in two steps.

\textbf{Step 1: Unique existence in $X_T$.}
For any $v_1$, $v_2\in X_T$, by \eqref{O_t-Lp} in Lemma \ref{lem.O_t}, we have that
\begin{align*}
\|B(v_1,v_2)\|_{L^{\infty}}&\lesssim_s \int^t_{0}(t-\tau)^{-\frac1{2s}}\tau^{\frac1s-2}\,d\tau\big(\sup_{0\leq t<T}t^{1-\frac1{2s}}\|v_1\|_{L^{\infty}}\big)\big(\sup_{0\leq t<T}t^{1-\frac1{2s}}\|v_2\|_{L^{\infty}}\big)\\
&\lesssim_st^{-1+1/(2s)}\big(\sup_{0\leq t<T}t^{1-1/(2s)}\|v_1\|_{L^{\infty}}\big)\big(\sup_{0\leq t<T}t^{1-1/(2s)}\|v_2\|_{L^{\infty}}\big).
\end{align*}
This gives that
\begin{equation}\label{est.23}
\sup_{0\leq t\leq T}t^{1-1/(2s)}\|B(v_1,v_2)\|_{L^{\infty}}\lesssim_s\big(\sup_{0\leq t<T}t^{1-1/(2s)}\|v_1\|_{L^{\infty}}\big)\big(\sup_{0\leq t<T}t^{1-1/(2s)}\|v_2\|_{L^{\infty}}\big).
\end{equation}
In addition, in view of Lemma \ref{lem.eqvi}, by Lemma \ref{lem.O_t} and Lemma \ref{lem.H}, we deduce that
\begin{equation}\label{eq.3}
\begin{split}
\|B(v_1,v_2)\|_{\bar{H}^{1}_{uloc}}
&\lesssim_s \int^t_{0}((t-\tau)^{-\frac1{2s}}+(t-\tau)^{-\frac3{4s}})\|v_1\otimes v_2\|_{\bar{H}^{1/2}_{uloc}}\,d\tau\\
&\lesssim_s\int^t_{0}((t-\tau)^{-\frac1{2s}}+(t-\tau)^{-\frac3{4s}})\|v_1\|_{\bar{H}^{1}_{uloc}}
\|v_2\|_{\bar{H}^{1}_{uloc}}\,\mathrm{d}\tau
\end{split}
\end{equation}
and
\begin{equation}\label{eq.4}
\begin{split}
\|\Lambda^{(3-3s)/2}\nabla B(v_1,v_2)\|_{E^{2}}
&\lesssim_s \int^t_{0}(t-\tau)^{-\frac34}\|v_1\otimes v_2\|_{\bar{H}^{7/2-3s}_{uloc}}\,\mathrm{d}\tau\\
&\lesssim_s \int^t_{0}(t-\tau)^{-\frac34}\|v_1\|_{\bar{H}^{(5-3s)/2}_{uloc}}
\|v_2\|_{\bar{H}^{(5-3s)/2}_{uloc}}\,\mathrm{d}\tau
\end{split}
\end{equation}
where we use that $7/2-3s<1$. These estimates give that
\begin{align}
&\|B(v_1,v_2)\|_{L^{4s/(4s-3)}([0,T];\bar{H}^1_{uloc})}\lesssim_s \|v_1\|_{L^{4s/(4s-3)}([0,T];\bar{H}^{1}_{uloc})}
\|v_2\|_{L^{4s/(4s-3)}([0,T];\bar{H}^{1}_{uloc})},\label{eq.6}\\
&\|\Lambda^{(3-3s)/2}\nabla B(v_1,v_2)\|_{L^4([0,T];E^{2})}
\lesssim_s \|v_1\|_{L^4([0,T];\bar{H}^{(5-3s)/2}_{uloc})}
\|v_2\|_{L^4([0,T];\bar{H}^{(5-3s)/2}_{uloc})}.\label{eq.4}
\end{align}
Collecting \eqref{eq.4} and \eqref{eq.6}, by H\"{o}lder's inequality and Lemma \ref{lem.eqvi}, we have that
\begin{equation}\label{eq.8}
\begin{split}
&\|B(v_1,v_2)\|_{L^4([0,T];\bar{H}^{(5-3s)/2}_{uloc})}\\
&\leq
\|\Lambda^{(3-3s)/2}\nabla B(v_1,v_2)\|_{L^4([0,T];E^{2})}+\|\nabla B(v_1,v_2)\|_{L^4([0,T];\bar{H}^{1}_{uloc})}\\
&\lesssim
\|\Lambda^{(3-3s)/2}\nabla B(v_1,v_2)\|_{L^4([0,T];E^{2})}+\|\nabla B(v_1,v_2)\|_{L^{4s/(4s-3)}([0,T];\bar{H}^{1}_{uloc})} \\
&\lesssim \|v_1\|_{L^{\frac{4s}{4s-3}}([0,T];\bar{H}^{1}_{uloc})}
\|v_2\|_{L^{\frac{4s}{4s-3}}([0,T];\bar{H}^{1}_{uloc})}+\|v_1\|_{L^4([0,T];\bar{H}^{\frac{5-3s}2}_{uloc})}
\|v_2\|_{L^4([0,T];\bar{H}^{\frac{5-3s}2}_{uloc})}.
\end{split}
\end{equation}
Combining with \eqref{est.23}, \eqref{eq.6} and \eqref{eq.8}, we get that for any $v_1,v_2\in X_T$ with $0<T\leq 1$
\begin{equation}\label{est.25'}
\|B(v_1,v_2)\|_{X_T}\leq C_1\|v_1\|_{X_T}\|v_2\|_{X_T}.
\end{equation}

On the other hand, since $\bar{H}^{5/2-2s}_{uloc}\hookrightarrow E_{3/(2s-1)}$,  by Lemma \ref{lem.G_t}, we get that
\begin{equation*}
\sup_{0\leq t\leq T}t^{1-1/(2s)}\|e^{-t\Lambda^{2s}}u_0\|_{L^{\infty}}\lesssim_s\|u_0\|_{\bar{H}^{5/2-2s}_{uloc}}\text{ and }\lim_{t\to 0+}t^{1-1/(2s)}\|e^{-t\Lambda^{2s}}u_0\|_{L^{\infty}}=0.
\end{equation*}
Moreover, adapting the splitting method used in Step 1 in the proof of Theorem \ref{TH1}, we deduce that for all $T\leq 1$
\[\|e^{- t\Lambda^{2s}}u_0\|_{ L^4([0,T];\bar{H}^{(5-3s)/2}_{uloc})\cap L^{4s/(4s-3)}([0,T];\bar{H}^{1}_{uloc})}\lesssim_s\|u_0\|_{\bar{H}^{5/2-2s}_{uloc}}.\]
Collecting all above three estimates, we get that
\begin{equation}\label{est.25}
\|e^{-t\Lambda^{2s}}u_0\|_{X_T}\leq C_2\|u_0\|_{\bar{H}^{5/2-2s}_{uloc}}\text{ and } \lim_{T\to 0+}\|e^{-t\Lambda^{2s}}u_0\|_{X_T}=0.
\end{equation}

According to \eqref{est.25}, we can find a time $T>0$ such that $4C_1\|e^{-t\Lambda^{2s}}u_0\|_{X_T}<1$. Hence, in view of  \eqref{est.25'}, by Banach fixed point theorem,  we conclude that \eqref{NS}-\eqref{NSI} admits a unique mild solution $u\in X_T$. In particular, if $u_0$ satisfies
$4C_1C_2\|u_0\|_{\bar{H}^{5/2-2s}_{uloc}}<1$, we can choose $T=1$.

\textbf{Step 2: Improved regularities.}
First, we show that $u\in L^2([0,T];\bar{H}^{5/2-s}_{uloc})$. It's obvious that
\[\|e^{-t\Lambda^{2s}}u_0\|_{L^2([0,T];\bar{H}^{5/2-s}_{uloc})}\leq C\|u_0\|_{\bar{H}^{5/2-2s}_{uloc}}.\]
On the other hand, along the exact same lines as in the proof of \eqref{est.1}, we get that
\begin{align*}
\|\Lambda^{3/2-s}\nabla B(u,u)\|_{E^2}\leq C_{\nu}\int^t_0(t-s)^{-1}\|u\|^2_{\bar{H}^{(5-3s)/2}_{uloc}}\,\mathrm{d}\tau
\leq\|u\|^2_{L^4([0,T];\bar{H}^{(5-3s)/2}_{uloc})}.
\end{align*}
Collecting the above two inequalities and \eqref{eq.6}, by H\"{o}lder's inequality, we deduce that $u\in L^2([0,T];\bar{H}^{5/2-s}_{uloc})$.

Next, we prove $u\in C([0,T];\bar{H}^{5/2-2s}_{uloc})$. For any $v_1\in C([0,T);\bar{H}^{5/2-2s}_{uloc})$ and $v_2\in X_T$, by Lemma \ref{lem.You}, we get that,
\begin{align*}
\|B(v_1,v_2)\|_{E^{2}}
\leq&C \int^t_{0}(t-\tau)^{-\frac1{2s}}\tau^{\frac1{2s}-1}\,\mathrm{d}\tau\|v_1\|_{L^{\infty}([0,T];E^2)}
\sup_{0\leq t\leq T}t^{1-1/(2s)}\|v_2\|_{L^{\infty}}\\
\leq &C\|v_1\|_{L^{\infty}([0,T];E^2)}
\sup_{0\leq t\leq T}t^{1-1/(2s)}\|v_2\|_{L^{\infty}}
\end{align*}
and
\begin{align*}
\|\Lambda^{5/2-2s}B(v_1,v_2)\|_{E_{2}}
\leq&C \int^t_{0}(t-\tau)^{-\frac1{2s}}\tau^{\frac1{2s}-1}\,\mathrm{d}\tau
\|v_2\|_{L^{\infty}([0,T];\bar{H}^{5/2-2s}_{uloc})}
\sup_{0\leq t\leq T}t^{1-\frac1{2s}}\|v_1\|_{L^{\infty}}\\
\leq &C\|v_2\|_{L^{\infty}([0,T];\bar{H}^{5/2-2s}_{uloc})}
\sup_{0\leq t\leq T}t^{1-1/({2s})}\|v_1\|_{L^{\infty}}.
\end{align*}
Hence we obtain that
\begin{equation}\label{est.26}
\|B(v_1,v_2)\|_{L^{\infty}([0,T];\bar{H}^{5/2-2s}_{uloc})}\leq C\|v_2\|_{L^{\infty}([0,T];\bar{H}^{5/2-2s}_{uloc})}
\sup_{0\leq t\leq T}t^{1-1/(2s)}\|v_2\|_{L^{\infty}},
\end{equation}
which implies that $B(v_1,v_2)\in C([0,T];\bar{H}^{5/2-2s}_{uloc})$. By Lemma \ref{lem.G_t}, we have that
\[\|e^{- t\Lambda^{2s}}u_0\|_{L^{\infty}([0,T];\bar{H}^{5/2-2s}_{uloc})}\leq C\|u_0\|_{\bar{H}^{5/2-2s}_{uloc}}.\]
Since $C^{\infty}_0(\RR^3)$ is dense in $\bar{H}^{5/2-2s}_{uloc}$, we have $e^{- t\Lambda^{2s}}u_0\in C([0,T];\bar{H}^{5/2-2s}_{uloc})$. Based on above analysis, we get that $u\in C([0,T];\bar{H}^{5/2-2s}_{uloc})$.
\end{proof}

\begin{corollary}\label{Cor-6.3}
If $u$ is the mild solution to \eqref{NS} with initial data $u_0\in\bar{H}^{5/2-2s}_{uloc}$ given in Theorem \ref{well-loc}, then, $u$ is a local Leray solution to \eqref{NS} starting from $u_0$.
\end{corollary}
\begin{proof}
From Theorem \ref{well-loc}, we know that $u\in C^{\infty}(\RR^3\times (0,T))$. In addition, due to $u\in C([0,T);E^p)$, we know that $P=\frac{\Div\Div}{\Delta}(u\otimes u)$ is well-defined on $\RR^3\times (0,T)$, which gives that $P\in L^{{3/2}}((0,T');L^{{3/2}}_{uloc})$ for any $T'<T$.
Hence, by elliptic theory and the fact that $u\in C^{\infty}(\RR^3\times (0,T))$, we deduce that
\[\|P(\cdot,t)\|_{C^{k,\alpha}(B_1)}\leq C_k\|u(\cdot,t)\otimes
u(\cdot,t)\|_{C^{k,\alpha}(B_2)}+C_k\|P(\cdot,t)\|_{L^{3/2}(B_2)},\]
which implies that $p$ is smooth on $\RR^3\times (0,T)$. So,
multiplying both sides of the following equation
\[\partial_t u+\Lambda^{2s}u+u\cdot\nabla u+\nabla P=0\]
by $u\varphi$, $\varphi\in \mathcal{D}(\RR^3\times (0,1))$, we derive that $u$ satisfies the local energy inequality \eqref{eq.loc}.
This, together with $u\in C_b([0,1];E^p)$ and $P\in  L^{{3/2}}((0,T');L^{{3/2}}_{uloc})$, gives that $u\in
L^{\infty}((0,T');L^2_{uloc})$ and $\Lambda^s u\in L^{2,2}_{uloc}(T')$ for any $T'<T$. Hence, we say that $u$ is a
local Leray solution to \eqref{NS}-\eqref{NSI} on $\RR^3\times (0,T)$ starting from $u_0$.
\end{proof}
\subsection{Leray theory of Perturbed fractional Navier-Stokes equations}
Similar to the Leary theory for perturbed classical Navier-Stokes equations developed in
Chapter 21 in \cite{LR02}, for the perturbed problem of the incompressible fractional
Navier-Stokes equations:
\begin{equation}\label{PFNS}\tag{PFNS}
\left\{\begin{array}{ll}
\partial_t w+\Lambda^{2s} w+w\cdot\nabla w+w\cdot\nabla v+v\cdot\nabla w+\nabla \bar{P}=0\\
\Div w=0,\\
w(x,0)=w_0,
\end{array}.\right.
\end{equation}
we also have the similar theory.
\begin{theorem}\label{TH.pertu}
Let  $v$ be a divergence free vector filed satisfying $\nabla v\in L^{q}([0,T];E^p)$ with $\tfrac{2s}q+\tfrac{3}p=2s-1$ and $p>3$. Then, for any divergence
free vector $w_0\in L^2$, there exists a weak solution $u$ to \eqref{PFNS} satisfying
\begin{enumerate}
  \item $w\in L^{\infty}([0,T];L^2)\cap L^2([0,T];\dot{H}^s)$;
  \item $\lim_{t\to 0+}\|w(t)-w_0\|_{L^2}=0$;
  \item $\|w(t)\|^2_{L^2}+2\int^t_0\|\Lambda^s w(\tau)\|^2_{L^2}\,\mathrm{d}\tau\leq
      \|w_0\|^2_{L^2}-2\int^t_0\int_{\RR^3}(w\cdot\nabla) v\cdot
      w\,\mathrm{d}x\mathrm{d}\tau;$
  \item for any $\varphi\in\mathcal{D}(\RR^3\times (0,T))$ and $\tilde{\varphi}\in \mathcal{D}(\RR^3)$ satisfying $\tilde{\varphi}\varphi=1$,
  \begin{equation}\label{6.1}
  \begin{split}
       &2\int^t_0\int_{\RR^3} \psi\big|\Lambda^s w\big|^2\,\mathrm{d}x\mathrm{d}\tau\\
       &\leq\int^t_0\int_{\RR^3}\big|w\big|^2\partial_t\psi\,\mathrm{d}x\mathrm{d}\tau
       -2\int^t_0\int_{\RR^3} [\tilde{\psi},\Lambda^{s}]\Lambda^s w\cdot
       w\psi\,\mathrm{d}x\mathrm{d}\tau\\
       &-2\int^t_0\int_{\RR^3}
       [\Lambda^s,\psi]w\cdot(\tilde{\psi}\Lambda^sw)\,\mathrm{d}x\mathrm{d}\tau+\int^t_0\int_{\RR^3}|w|^2
       (w+v)\cdot\nabla\psi\,\mathrm{d}x\mathrm{d}\tau\\
       &-2\int^t_0\int_{\RR^3}(w\cdot\nabla) v\cdot w\psi\,\mathrm{d}x\mathrm{d}\tau+2\int^t_0\int_{\RR^3}\bar{P}\cdot u\nabla\psi\,\mathrm{d}x\mathrm{d}\tau.
      \end{split}
      \end{equation}
\end{enumerate}
\end{theorem}
\begin{proof}
We first give out a key priori estimate. Multiplying both sides of the first equation of \eqref{PFNS} by $w$ and integrating over
$\RR^3\times (0,t)$, we get that
\begin{align*}
\|w(t)\|^2_{L^2}+2\int^t_0\int_{\RR^3}|\Lambda^s
w|^2\,\mathrm{d}x\mathrm{d}\mathrm{\tau}=\|w_0\|_{L^2}-2\int^t_0\int_{\RR^3}(w
\cdot\nabla)v\cdot w\,\mathrm{d}x\mathrm{d}\mathrm{\tau}.
\end{align*}
Let $X_r$ be the pointwise multiplier space
of negative order defined in Chapter 21 in \cite{LR02}. By Theorem 21.1 in \cite{LR02}, we know that
$E^{\frac3r}\hookrightarrow X_{r}$ for any $r\in (0,1]$. So, we have that
\[\|w\cdot\nabla v\|_{L^2}\leq C\|w\|_{H^{3/p}}\|\nabla v\|_{E^{p}}.\]
Hence, we find that
\begin{align*}
&\|w(t)\|^2_{L^2}+2\int^t_0\int_{\RR^3}|\Lambda^s
w|^2\,\mathrm{d}x\mathrm{d}\mathrm{\tau}\\
&\leq\|w_0\|_{L^2}
+C\int^t_0\|w\|_{L^2}\|w\|_{H^{3/p}}\|\nabla v\|_{E^{p}}\,\mathrm{d}\tau\\
&\leq\|w_0\|_{L^2}
+C\int^t_0\|w\|_{L^2}\big(\|w\|_{L^2}+\|w\|^{1-\frac3{ps}}_{L^2}
\|w\|^{\frac3{ps}}_{\dot{H}^{s}}\big)\|\nabla v\|_{E^{p}}\,\mathrm{d}\tau\\
&\leq\|w_0\|_{L^2}+\int^t_0\int_{\RR^3}|\Lambda^s
w|^2\,\mathrm{d}x\mathrm{d}\mathrm{\tau}
+C\int^t_0\|w\|^2_{L^2}\big(\|\nabla v\|_{E^{p}}+\|\nabla v\|^{\frac{2ps}{2ps-3}}_{E^p}\big)\,\mathrm{d}\tau.
\end{align*}
Then, by Gronwall's inequality, we obtain that for any $0\leq t\leq T$
\begin{equation}\label{est-6-4}
\begin{split}
&\|w(t)\|^2_{L^2}+\int^t_0\int_{\RR^3}|\Lambda^s w|^2\,\mathrm{d}x\mathrm{d}\mathrm{\tau}\\
&\leq C\|w_0\|^2_{L^2}\exp\Big\{t^{1-\frac1{2s}}\big(\sup_{0\leq t\leq T}t^{\frac1{2s}}\|\nabla
v\|_{E^{p}}\big)+t^{1-\frac{p}{2sp-3}}\big(\sup_{0\leq t\leq T}t^{\frac1{2s}}\|\nabla
v\|_{E^{p}}\big)^{\frac{2ps}{2ps-3}}\Big\}.
\end{split}
\end{equation}

Then, based on the priori estimate \eqref{est-6-4}, following the argument used in Chapter 21 in \cite{LR02} to deal with the perturbed classical Navier-Stokes equations, we can prove this theorem. Since the proof is completely parallel to \cite{LR02}, we omit it here.
\end{proof}

\end{document}